\documentclass[11pt,a4paper]{article}
\usepackage{amssymb}
\usepackage{amsmath}
\usepackage{latexsym}
\usepackage{amsthm}
\usepackage{bbm}
\usepackage{mathrsfs}
\usepackage{xcolor}
\usepackage{cite}

\newtheorem{thm}{Theorem}
\newtheorem{lem}[thm]{Lemma}

\newtheorem{cor}[thm]{Corollary}
\newtheorem{definition}{Definition}
\theoremstyle{definition}
\newtheorem{remark}{Remark}

\usepackage{xpatch}
\xpatchcmd{\proof}{\itshape}{\normalfont\proofnameformat}{}{}
\newcommand{\proofnameformat}{}

\begin{document}

\renewcommand{\proofnameformat}{\bfseries}

\begin{center}
{\Large\textbf{Riesz Energy, $L^2$ Discrepancy, and Optimal Transport of Determinantal Point Processes on the Sphere and the Flat Torus}}

\vspace{10mm}

\textbf{Bence Borda$^1$, Peter Grabner$^1$, and Ryan W. Matzke$^{1,2}$}

\vspace{3mm}

{\footnotesize $^1$ Institute of Analysis and Number Theory, Graz University of Technology

Steyrergasse 30, 8010 Graz, Austria

\vspace{3mm}

$^2$Vanderbilt University, Department of Mathematics

1326 Stevenson Center, Nashville, TN 37240, USA\\

Email: \texttt{borda@math.tugraz.at}, \texttt{peter.grabner@tugraz.at}, \texttt{ryan.w.matzke@vanderbilt.edu}}

\vspace{5mm}

{\footnotesize \textbf{Keywords:} harmonic ensemble, spherical ensemble, spherical cap discrepancy, $L^2$ discrepancy, Wasserstein metric}

{\footnotesize \textbf{Mathematics Subject Classification (2020):} 60G55, 31C12, 49Q22, 11K38}
\end{center}

\vspace{5mm}

\begin{abstract}
Determinantal point processes exhibit an inherent repulsive behavior, thus providing examples of very evenly distributed point sets on manifolds. In this paper, we study the so-called harmonic ensemble, defined in terms of Laplace eigenfunctions on the sphere $\mathbb{S}^d$ and the flat torus $\mathbb{T}^d$, and the so-called spherical ensemble on $\mathbb{S}^2$, which originates in random matrix theory. We extend results of Beltr\'an, Marzo and Ortega-Cerd\`a on the Riesz $s$-energy of the harmonic ensemble to the nonsingular regime $s<0$, and as a corollary find the expected value of the spherical cap $L^2$ discrepancy via the Stolarsky invariance principle. We find the expected value of the $L^2$ discrepancy with respect to axis-parallel boxes and Euclidean balls of the harmonic ensemble on $\mathbb{T}^d$. We also show that the spherical ensemble and the harmonic ensemble on $\mathbb{S}^2$ and $\mathbb{T}^2$ with $N$ points attain the optimal rate $N^{-1/2}$ in expectation in the Wasserstein metric $W_2$, in contrast to i.i.d.\ random points, which are known to lose a factor of $(\log N)^{1/2}$.
\end{abstract}

\section{Introduction}

Point processes with built-in repulsion are classical models for fermionic particle systems in quantum mechanics. A similar phenomenon of repulsion occurs for the eigenvalues of random matrices and the zeros of random polynomials (see, e.g., \cite[Chapter 4]{HKPV1}). Many repulsive systems fall into the category of determinantal point processes: these are simple point processes whose joint intensities are given by the determinant of a suitable positive semidefinite kernel. Such systems provide a natural way of distributing points on a compact manifold, with the built-in repulsion ensuring a more even point distribution than, say, independent and identically distributed (i.i.d.) random variables. In Section \ref{sec:Results}, we give the precise definitions of all notions appearing here in the Introduction.

The main goal of this paper is to demonstrate the repulsive properties of the so-called \emph{harmonic ensemble} on the unit sphere $\mathbb{S}^d$ and the flat torus $\mathbb{T}^d$, and the so-called \emph{spherical ensemble} on $\mathbb{S}^2$. In particular, these systems are closer to uniformity in terms of the nonsingular Riesz energy, in $L^2$ discrepancy and in the Wasserstein metric $W_2$ than uniformly distributed i.i.d.\ points.

The harmonic ensemble on a compact manifold $M$ is the determinantal point process defined by the reproducing kernel of a finite dimensional subspace of $L^2 (M)$, typically that spanned by the Laplace eigenfunctions corresponding to eigenvalues up to a certain threshold. It was first introduced by Beltr\'an, Marzo, and Ortega-Cerd\`a on the sphere $\mathbb{S}^d$ in \cite{BMO}, where they established the precise asymptotics of its expected singular Riesz and logarithmic energies, and studied its separation distance and linear statistics. The energy results have since been extended to projective spaces \cite{ADGMS} and to $\mathrm{SO}(3)$ \cite{BF}, where the Green energy was also found. The harmonic ensemble has also been considered on the flat torus in \cite{MOC}, where its expected periodic Riesz energy for $0<s<d$ was investigated. The harmonic ensemble on both the sphere \cite{BGKZ} and the torus \cite{ST} has been shown to be hyperuniform.

The spherical ensemble has its origins in random matrix theory. Krishnapur \cite{KR1,KR2} showed that the complex eigenvalues of $A^{-1}B$, where $A$ and $B$ are independent $N \times N$ matrices with i.i.d. standard complex Gaussian entries, form a determinantal point process on $\mathbb{C}$ with respect to a certain background measure. The spherical ensemble is then obtained by the stereographic projection of these eigenvalues onto the sphere $\mathbb{S}^2$, and turns out to be a determinantal point process itself. Its precise distribution properties were investigated in \cite{AZ}, in particular both the singular and nonsingular Riesz energy and the $L^2$ discrepancy were shown to be of optimal order among all point sets of the same size. The spherical ensemble was generalized to complex projective spaces and even-dimensional spheres in \cite{BE2,BE1}.

For the empirical spectral measure of further random matrix models, we refer to the monograph \cite{BS}. In particular, estimating the convergence rate in the Kolmogorov and the Wasserstein metric on the circle (for random unitary matrices), on the real line (for random Hermitian matrices) and on the complex plane (for random non-Hermitian matrices) is a classical problem. The latter metric originates in the theory of optimal transport, and has become a powerful tool in geometric analysis, see \cite{VI} for an introduction. We follow the approach of G\"otze and Jalowy \cite{GJ,JA} to deduce sharp rates in the Wasserstein metric by using a suitable smoothing inequality.

The main contributions of the present paper, discussed in more detail in Section \ref{sec:Results}, are the following.
\begin{itemize}
\item We find the expected value of the nonsingular Riesz energy of the harmonic ensemble on $\mathbb{S}^d$ (Theorem \ref{thm:riesztheorem}).

\item We find the expected value of the $L^2$ discrepancy of the harmonic ensemble on $\mathbb{S}^d$ via the Stolarsky invariance principle (Corollary \ref{cor:capL2corollary Sd}), and show that it has optimal order up to a factor of $(\log N)^{1/2}$ among all point sets of the same size.

\item We find the expected value of the $L^2$ discrepancy with respect to axis-parallel boxes and Euclidean balls of the harmonic ensemble on $\mathbb{T}^d$ via Fourier analysis (Theorems \ref{L2torustheorem} and \ref{ballL2torustheorem}). The $L^2$ discrepancy with respect to Euclidean balls has optimal order up to a factor of $(\log N)^{1/2}$ among all point sets of the same size.

\item We show that the harmonic ensemble on $\mathbb{S}^2$ and $\mathbb{T}^2$, and the spherical ensemble on $\mathbb{S}^2$ have distance $O(N^{-1/2})$ from uniformity in expectation in the Wasserstein metric $W_2$, which is optimal among all point sets of the same size (Theorem \ref{thm: 2dimWassersteintheorem}).
\end{itemize}

The proof of all our results rely on the explicit formulas for the Laplace eigenfunctions on $\mathbb{S}^d$ and $\mathbb{T}^d$. We mention without further details that similarly to \cite{ADGMS}, our approach could be extended to other two-point homogeneous spaces, on which the Laplace eigenfunctions are also explicitly known. We only comment on the complex projective space $\mathbb{CP}^{d}$, in particular on the expected value of the $L^2$ discrepancy of the projective ensemble (Corollary \ref{cor:capL2corollary CP}).

\section{Results}\label{sec:Results}

\subsection{Determinantal point processes}

Throughout the paper, $M$ denotes a $d$-dimensional compact, connected, smooth Riemannian manifold without boundary, with Riemannian volume measure $\mathrm{Vol}$. Let $K: M \times M \to \mathbb{C}$ be a continuous, Hermitian, positive semidefinite kernel; that is, for any $x_1, x_2, \ldots, x_k \in M$, the $k \times k$ matrix $(K(x_i,x_j))_{i,j=1}^k$ is self-adjoint and positive semidefinite. Assume further that the operator $L_K: L^2(M) \rightarrow L^2(M)$ defined by
\begin{equation*}
L_K(f) (x) := \int_{M} K(x,y) f(y) \,\mathrm{d} \mathrm{Vol}(y)
\end{equation*}
has all eigenvalues in $[0,1]$. Let $(\lambda_n, \psi_n)$ be the sequence of eigenvalues and normalized eigenfunctions. Then Mercer's Theorem tells us that
\begin{equation*}
K(x,y) = \sum_{n=1}^{\infty} \lambda_n \psi_n(x) \overline{\psi_n(y)},
\end{equation*}
where convergence is absolute and uniform. This immediately implies that
\begin{equation*}
\mathrm{Trace}(L_K) = \sum_{n=1}^{\infty} \lambda_n < \infty,
\end{equation*}
meaning $L_K$ is trace class.

A simple point process $\mathscr{X}_K$ on $M$ is called a determinantal point process with kernel $K$ if the joint intensities (correlation functions) with respect to $\mathrm{Vol}$ are $\rho_k (x_1, x_2, \ldots, x_k) = \det (K(x_i,x_j))_{i,j=1}^k$, $k=1,2,\ldots$. For every eigenfunction $\psi_n$, the process produces a point with probability $\lambda_n$, which implies that the expected number of points is the trace of $L_K$. If $\lambda_1 = \cdots = \lambda_N = 1$, and $\lambda_n = 0$ for all $n > N$, then $L_K$ is a projection to the subspace spanned by $\psi_1, \ldots, \psi_N$. In this case we call $K$ a projection kernel, and the corresponding determinantal point process  samples exactly $N$ points, $\mathscr{X}_K = \{X_1, \ldots, X_N\}$. For the rest of the paper, we will only deal with projection kernels. In particular, for any Borel measurable function $f: M \to \mathbb{C}$,
\begin{equation}\label{expectationformula}
\mathbb{E} \sum_{n=1}^{N} f(X_n) = \int_M K(x,x) f(x) \, \mathrm{d} \mathrm{Vol} (x) ,
\end{equation}
and
\begin{equation}\label{varianceformula}
\begin{split} \mathbb{E} \left| \sum_{n=1}^N f(X_n) \right|^2 = &\mathbb{E} \sum_{n=1}^N |f(X_n)|^2 + \mathbb{E} \sum_{\substack{n,m=1 \\ n \neq m}}^N f(X_n) \overline{f(X_m)} \\ = &\int_M K(x,x) |f(x)|^2 \, \mathrm{d} \mathrm{Vol} (x) \\ &+ \int_M \int_M (K(x,x) K(y,y) - |K(x,y)|^2) f(x) \overline{f(y)} \, \mathrm{d} \mathrm{Vol}(x) \, \mathrm{d} \mathrm{Vol} (y) , \end{split}
\end{equation}
provided that the integrals exist. For the general theory of determinantal point processes, we refer to \cite{HKPV}.

In this paper, we will work on the standard Euclidean unit sphere $M=\mathbb{S}^d \subset \mathbb{R}^{d+1}$ and the flat torus $M= \mathbb{T}^d = \mathbb{R}^d/\mathbb{Z}^d$. The volumes are normalized as $\mathrm{Vol}(\mathbb{S}^d)=2 \pi^{\frac{d+1}{2}}/\Gamma \left( \frac{d+1}{2} \right)$ and $\mathrm{Vol}(\mathbb{T}^d)=1$. Recall the orthogonal decomposition $L^2 (\mathbb{S}^d, \mathrm{Vol}) = \oplus_{\ell =0}^{\infty} H_{\ell}$, where $H_{\ell}$ is the space of spherical harmonics of degree $\ell$. Let $Y_{\ell}^m$, $1 \le m \le \mathrm{dim} \, H_{\ell}=\frac{2 \ell +d-1}{\ell +d-1} \binom{\ell +d-1}{d-1}$ be an (arbitrary) orthonormal basis in $H_{\ell}$. We refer to \cite{SW} for a general overview of spherical harmonics.
\begin{definition} \hspace{1mm}
\begin{enumerate}
\item[(i)] Let $L \in \mathbb{N}$. The harmonic ensemble on $\mathbb{S}^d$ is the determinantal point process $X=X(L,\mathbb{S}^d)=\{ X_1, X_2, \ldots, X_N \}$ with kernel $K: \mathbb{S}^d \times \mathbb{S}^d \to \mathbb{C}$,
\begin{equation*}
K(x,y) = \sum_{\ell =0}^{L} \sum_{m=1}^{\mathrm{dim}(H_{\ell})} Y_{\ell}^m (x) \overline{Y_{\ell}^m (y)} .
\end{equation*}
\item[(ii)] Let $T \in \mathbb{N}$. The harmonic ensemble on $\mathbb{T}^d$ is the determinantal point process $X=X(T,\mathbb{T}^d)=\{ X_1, X_2, \ldots, X_N \}$ with kernel $K: \mathbb{T}^d \times \mathbb{T}^d \to \mathbb{C}$,
\begin{equation*}
K(x,y) = \sum_{\substack{k \in \mathbb{Z}^d \\ \| k \|_{\infty} \le T}} e^{2 \pi i \langle k,x-y \rangle} . 
\end{equation*} 
\end{enumerate}
\end{definition}
Note that $K(x,y)$ on $\mathbb{S}^d$, resp.\ $\mathbb{T}^d$, is the reproducing kernel of the finite dimensional subspace $\oplus_{\ell =0}^L H_{\ell} \le L^2 (\mathbb{S}^d, \mathrm{Vol})$, resp.\ $\mathrm{Span} \, \{ e^{2 \pi i \langle k,x \rangle} \, : \, \| k \|_{\infty} \le T \} \le L^2 (\mathbb{T}^d, \mathrm{Vol})$. The kernel is rotation invariant on $\mathbb{S}^d$, resp.\ translation invariant on $\mathbb{T}^d$. By the above discussion, the number of points in     the harmonic ensemble is deterministic, and the number of points is given by 
\begin{equation*}
N = \begin{cases}
 \frac{2L+d}{d} \binom{L+d-1}{d-1} & \textrm{if } M=\mathbb{S}^d, \\ (2T+1)^d & \textrm{if } M=\mathbb{T}^d .  \end{cases}
\end{equation*}

\begin{remark}\label{harmoniconM} More generally, we could define the harmonic ensemble on a compact manifold $M$ as the determinantal point process $X=\{ X_1, X_2, \ldots, X_N \}$ with kernel $K: M \times M \to \mathbb{C}$, $K(x,y) = \sum_{k=0}^{N-1} \phi_k (x) \overline{\phi_k (y)}$, where $\phi_k$, $k \ge 0$ is an orthonormal basis in $L^2(M, \mathrm{Vol})$ that consists of eigenfunctions of the Laplace--Beltrami operator on $M$, in increasing order of the eigenvalues.
\end{remark}

\begin{definition} Let $N \in \mathbb{N}$, and consider two independent $N \times N$ matrices $A,B$ whose entries are i.i.d.\ mean zero Gaussian random variables with density $\pi^{-1} e^{-|z|^2}$ on $\mathbb{C}$. The spherical ensemble on $\mathbb{S}^2$ is the stereographic projection of the $N$ complex eigenvalues of the matrix $A^{-1}B$ to $\mathbb{S}^2$.
\end{definition}

The complex eigenvalues of $A^{-1}B$ form a determinantal point process on $\mathbb{C}$ with kernel $K(z,w) = (1+z \overline{w})^{N-1}$ with respect to the measure $\mathrm{d}\mu (z) = \frac{N}{\pi (1+|z|^2)^{N+1}} \mathrm{d} z$ \cite{KR2,KR1}. After stereographic projection, the spherical ensemble also becomes a determinantal point process on $\mathbb{S}^2$ \cite{BE1,BGKZ}, with kernel $K: \mathbb{S}^2 \times \mathbb{S}^2 \to \mathbb{C}$,
\begin{equation*}
K(x,y) = \frac{N}{\pi 2^{N+1}} \left( \frac{1+\langle x,y \rangle -x_3-y_3 + i (x_2 y_1 - x_1 y_2)}{\sqrt{(1-x_3)(1-y_3)}} \right)^{N-1}.
\end{equation*}
In particular,
\begin{equation}\label{sphericalkernel}
K(x,x)=\frac{N}{4 \pi} \qquad \textrm{and} \qquad |K(x,y)|^2 = \frac{N^2}{16 \pi^2} \left( \frac{1+\langle x,y \rangle}{2} \right)^{N-1} .
\end{equation}
Obviously, the number of points in the spherical ensemble is also deterministic.

\subsection{Riesz energy on the sphere}

The continuous Riesz $s$-energy of a Borel probability measure $\mu$ on $\mathbb{S}^d$ is defined as
\[ E_{s}(\mu) = \left\{ \begin{array}{ll} \displaystyle{\int_{\mathbb{S}^d} \int_{\mathbb{S}^d} \frac{1}{|x-y|^s} \, \mathrm{d}\mu(x) \, \mathrm{d}\mu(y)} & \textrm{if } s \neq 0, \\ \displaystyle{\int_{\mathbb{S}^d} \int_{\mathbb{S}^d} \log \frac{1}{|x-y|} \, \mathrm{d}\mu(x) \, \mathrm{d}\mu(y)} & \textrm{if } s=0, \end{array} \right. \]
where $| \cdot |$ denotes the Euclidean norm. For the normalized Riemannian volume, $E_s(\mathrm{Vol}/\mathrm{Vol}(\mathbb{S}^d))$ is finite if and only if $s<d$, and we set
\[ \mathcal{I}_{s,d} = E_s(\mathrm{Vol}/\mathrm{Vol}(\mathbb{S}^d)) = \left\{ \begin{array}{ll} 2^{d-s-1} \frac{\Gamma \left( \frac{d+1}{2} \right) \Gamma \left( \frac{d-s}{2} \right)}{\sqrt{\pi} \Gamma \left( d-\frac{s}{2} \right)} & \textrm{if } s<d, \,\, s \neq 0, \\ \frac{\psi_0 (d) - \psi_0 (d/2)}{2} - \log 2 & \textrm{if } s=0, \end{array} \right. \]
where $\psi_0 (x) = \Gamma'(x) / \Gamma (x)$ is the digamma function.

The discrete Riesz $s$-energy of a finite point set $A_N = \{ a_1, a_2, \ldots, a_N \} \subset \mathbb{S}^d$ is similarly defined as
\[ E_s (A_N) = \sum_{\substack{n,m =1 \\ n \neq m}}^N \frac{1}{|a_n-a_m|^s} \quad (s \neq 0) \qquad \textrm{and} \qquad E_0 (A_N) = \sum_{\substack{n,m =1 \\ n \neq m}}^N \log \frac{1}{|a_n-a_m|} . \]
Note that in the nonsingular regime $s<0$, $E_s(A_N)$ is $N^2$ times the continuous Riesz $s$-energy of the corresponding empirical measure $\mu=N^{-1} \sum_{n=1}^N \delta_{a_n}$.

The optimization of the discrete and continuous Riesz energies on the sphere has been well-studied \cite{BHS}. Bj\"{o}rck \cite{Bjo} showed that
\begin{itemize}
\item if $-2 < s < 0$, then $E_{s}(\mu)$ is uniquely maximized by the normalized volume measure;
\item if $s = -2$, then the maximizers of $E_{s}(\mu)$ are exactly the measures with center of mass at the origin, i.e. $\int_{\mathbb{S}^d} x \, \mathrm{d}\mu(x) = 0$;
\item if $s < -2$, then the maximizers of $E_{s}(\mu)$ are exactly the measures of the form $\frac{1}{2}( \delta_p + \delta_{-p})$ for any $p \in \mathbb{S}^d$.
\end{itemize}

This immediately yields maximizers of the discrete Riesz energies $E_s (A_N)$ for $s \leq -2$. In the case $s = -2$, any point set with center of mass at the origin is optimal, e.g.\ evenly spaced points on some great circle. In the case when $s < -2$ and $N$ is even, taking $N/2$ points at a pole and the rest at the opposite pole is optimal.

In the case $-2 < s < 0$, we have (see \cite[Theorem 1.1]{BraHS}, \cite[Theorem 2(a)]{Wag1}, \cite[Theorem B]{Wag2})
\begin{equation}\label{eq:Max Riesz Energy Asymptotics}
\mathcal{I}_{s,d} N^2 - c_{s,d} N^{1+\frac{s}{d}} \le \max_{|A_N| = N, \, A_N \subset \mathbb{S}^d} E_{s}(A_N) \le \mathcal{I}_{s,d} N^2 -  c'_{s,d} N^{1+\frac{s}{d}}
\end{equation}
with some positive constants $c_{s,d}$, $c'_{s,d}$. The optimal point configurations $A_N$ are not known in general, however it is known that they are uniformly distributed as $N \rightarrow \infty$ \cite[Theorem 6.1.7]{BHS}.

Consider now the Riesz $s$-energy $E_s (X)$ of the harmonic ensemble
$X=X(L,\mathbb{S}^d)=\{ X_1, X_2, \ldots, X_N \}$. Beltr\'an, Marzo, and
Ortega-Cerd\`a \cite{BMO} showed that its expected value is finite if and only
if $s<d+2$, and gave an explicit formula in the singular regime $0 \le s \le
d$. The fact that the expectation of the Riesz energy under such determinantal
point processes converges for $s<d+2$ (as opposed to $s<d$ for i.i.d.\ points)
can be explained by the vanishing to the second order of the joint intensity $\rho_2(x,y)$ for $x=y$.

In particular, for $0<s<d$ they showed that
\begin{equation*} \mathbb{E} E_s (X) = \mathcal{I}_{s,d} N^2 - (1+o(1)) C_{s,d} N^{1+\frac{s}{d}} 
\end{equation*}
with
\begin{equation}\label{Csd}
\begin{split} C_{s,d} &= 2^{s-\frac{s}{d}} \mathcal{I}_{s,d} \frac{d}{(d!)^{1-\frac{s}{d}}} \cdot \frac{\Gamma \left( 1 + \frac{d}{2} \right) \Gamma \left( \frac{1+s}{2} \right) \Gamma \left( d - \frac{s}{2} \right)}{\sqrt{\pi} \Gamma \left( 1 + \frac{s}{2} \right) \Gamma \left( 1+\frac{d+s}{2} \right)} \\ &= \frac{(d!)^{\frac{s}{d}} \Gamma \left( \frac{1+s}{2} \right) \Gamma \left( \frac{d-s}{2} \right)}{\sqrt{\pi} 2^{\frac{s}{d}} \left( 1+\frac{s}{d} \right) \Gamma \left( 1+\frac{s}{2} \right) \Gamma \left( \frac{d+s}{2} \right)} . \end{split}
\end{equation}
In this paper, we extend their result to the nonsingular regime $s<0$.
\begin{thm}\label{thm:riesztheorem}
The harmonic ensemble $X=X(L,\mathbb{S}^d)=\{ X_1, X_2, \ldots, X_N \}$ satisfies
\[ \mathbb{E} E_s (X) = \mathcal{I}_{s,d} N^2- \left\{ \begin{array}{ll} C_{s,d} N^{1+\frac{s}{d}} +O \left( N^{1-\frac{1}{d}} \right) & \textrm{if } -1<s<0, \\ C_{-1,d} N^{1-\frac{1}{d}} \log N + \kappa_d N^{1-\frac{1}{d}} + O \left( N^{1-\frac{2}{d}} \log N \right) & \textrm{if } s=-1, \\ C_{s,d} N^{1-\frac{1}{d}} + O \left( N^{1+\frac{\max\{ s, -2 \}}{d}} \right) & \textrm{if } s<-1 \end{array} \right. \]
with implied constants depending only on $s$ and $d$, where
\begin{equation*} C_{s,d} = \left\{ \begin{array}{ll} \textrm{as in \eqref{Csd}} & \textrm{if } -1<s<0, \\ \frac{2^{\frac{1}{d}}}{\pi (d!)^{\frac{1}{d}}} & \textrm{if } s=-1, \\ \frac{d 2^{\frac{1}{d}-s-2} \Gamma \left( -\frac{1+s}{2} \right)}{\sqrt{\pi} (d!)^\frac{1}{d} \Gamma \left( -\frac{s}{2} \right)} & \textrm{if } s<-1 , \end{array} \right. 
\end{equation*}
and
\begin{equation*} 
\kappa_d = \frac{2^{\frac{1}{d}}}{\pi (d!)^{\frac{1}{d}}} \left( \log \frac{d!}{2} + (4+(-1)^d) d \log 2 -2d \sum_{j=1}^{\lfloor d/2 \rfloor} \frac{1}{d-2j+1} \right) . 
\end{equation*}
\end{thm}
Comparing to \eqref{eq:Max Riesz Energy Asymptotics}, we thus see that the harmonic ensemble has optimal second order asymptotics in expectation for $-1 < s < 0$, but not for $-2 < s \leq -1$.

The expected value of the Riesz $s$-energy $E_s (Z)$ of the spherical ensemble $Z=\{ Z_1, Z_2, \ldots, Z_N \}$ was found by Alishahi and Zamani \cite{AZ} in the whole range $s<4$. In particular, for $s<4$, $s \neq 0,2$,
\begin{equation*}
 \mathbb{E} E_s (Z) = \mathcal{I}_{s,2} N^2 - \frac{\Gamma \left( 1-\frac{s}{2} \right)}{2^s} \cdot \frac{N^2 \Gamma \left( N \right)}{\Gamma \left( N+1-\frac{s}{2} \right)} =\mathcal{I}_{s,2} N^2 - (1+o(1)) \frac{\Gamma \left( 1-\frac{s}{2} \right)}{2^s} N^{1+\frac{s}{2}} , 
\end{equation*}
thus the spherical ensemble has optimal second order asymptotics for $-2 < s < 0$.

In comparison, for $N$ uniformly distributed i.i.d.\ points $Y=\{ Y_1, Y_2, \ldots, Y_N \}$ on $\mathbb{S}^d$ we have $\mathbb{E} E_s(Y) = (N^2-N) \mathcal{I}_{s,d}$ if $s<d$, and $\mathbb{E} E_s(Y)=\infty$ if $s \ge d$.

\subsection{$L^2$ discrepancy on the sphere}

The sets
\begin{equation*}
 C(x,t) := \left\{ y \in \mathbb{S}^d \, : \, \langle x,y \rangle \ge t \right\} , \qquad x \in \mathbb{S}^d, \,\, t \in [-1,1] 
 \end{equation*}
are called spherical caps in $\mathbb{S}^d$. The spherical cap $L^2$ discrepancy of a finite point set $A_N=\{ a_1, a_2, \ldots, a_N \} \subset \mathbb{S}^d$ is defined as
\[ D_{\mathrm{cap},2}(A_N) = \left( \int_{-1}^1 \frac{1}{\mathrm{Vol}(\mathbb{S}^d)} \int_{\mathbb{S}^d} \left| \frac{|\{ 1 \le n \le N \, : \, a_n \in C(x,t) \}|}{N} - \frac{\mathrm{Vol}(C(x,t))}{\mathrm{Vol}(\mathbb{S}^d)} \right|^2 \, \mathrm{d} \mathrm{Vol}(x) \, \mathrm{d} t \right)^{1/2} . \]
The smallest possible order for any set of $N$ points is (see \cite[Proposition 1.1]{Ale}, \cite[Theorem 1]{Sto}, \cite{Beck2})
\[ \frac{1}{N^{\frac{1}{2} + \frac{1}{2d}}} \ll \inf_{|A_N|=N, \, A_N \subset \mathbb{S}^d} D_{\mathrm{cap},2}(A_N) \ll \frac{1}{N^{\frac{1}{2} + \frac{1}{2d}}} . \]
The spherical cap $L^2$ discrepancy relates to the Riesz $(-1)$-energy via the Stolarsky invariance principle
\[ N^2 D_{\mathrm{cap,2}}^2 (A_N) = \frac{\Gamma \left( \frac{d+1}{2} \right)}{\sqrt{\pi}d \Gamma \left( \frac{d}{2} \right)} \left( \mathcal{I}_{-1,d}N^2 - E_{-1}(A_N) \right) . \]
Stolarsky's original proof in \cite{Sto} has been simplified in \cite{BDM, BD}, and the result has been generalized to projective spaces by Skriganov \cite{SKR3}. We refer the reader to \cite{SKR1, SKR4} for a detailed study of the $L^p$ discrepancy on compact metric spaces, and to \cite{DT} for a general exposition on discrepancy.

Theorem \ref{thm:riesztheorem} with $s=-1$ and the Stolarsky invariance principle thus immediately imply that the expected value of the spherical cap $L^2$ discrepancy of the harmonic ensemble has optimal order up to a factor of $(\log N)^{1/2}$.
\begin{cor}\label{cor:capL2corollary Sd} 
The harmonic ensemble $X=X(L,\mathbb{S}^d)=\{ X_1, X_2, \ldots, X_N \}$ satisfies
\begin{equation*}
\mathbb{E} D_{\mathrm{cap,2}}^2 (X) = \frac{2^{\frac{1}{d}} \Gamma \left( \frac{d+1}{2} \right)}{\pi^{\frac{3}{2}} d (d!)^{\frac{1}{d}} \Gamma \left( \frac{d}{2} \right)} \cdot \frac{\log N}{N^{1+\frac{1}{d}}} + \frac{\kappa_d \Gamma \left( \frac{d+1}{2} \right)}{\pi^{\frac{1}{2}} d \Gamma \left( \frac{d}{2} \right)} \cdot \frac{1}{N^{1+\frac{1}{d}}} + O \left( \frac{\log N}{N^{1+\frac{2}{d}}} \right)
\end{equation*}
with an implied constant depending only on $d$. In particular,
\begin{equation*}
\sqrt{\mathbb{E} D_{\mathrm{cap,2}}^2 (X)} \sim \frac{2^{\frac{1}{2d}} \Gamma \left( \frac{d+1}{2} \right)^{\frac{1}{2}}}{\pi^{\frac{3}{4}} d^{\frac{1}{2}} (d!)^{\frac{1}{2d}} \Gamma \left( \frac{d}{2} \right)^{\frac{1}{2}}} \cdot \frac{(\log N)^{\frac{1}{2}}}{N^{\frac{1}{2}+\frac{1}{2d}}} .
\end{equation*}
\end{cor}

The spherical ensemble $Z=\{ Z_1, Z_2, \ldots, Z_N \}$ on $\mathbb{S}^2$ attains the optimal rate \cite[Remark 4.3]{AZ}:
\[ \sqrt{\mathbb{E} D_{\mathrm{cap,2}}^2 (Z)} = \sqrt{\frac{\Gamma \left( \frac{3}{2} \right) \Gamma \left( N \right)}{\Gamma \left( N+\frac{3}{2} \right)} } \sim \frac{\pi^{1/4}}{2^{1/2} N^{3/4}} . \]
In comparison, for $N$ uniformly distributed i.i.d.\ points $Y=\{ Y_1, Y_2, \ldots, Y_N \}$ on $\mathbb{S}^d$, we have $\sqrt{\mathbb{E} D_{\mathrm{cap},2}^2 (Y)} = \left( \frac{\Gamma ((d+1)/2) \mathcal{I}_{-1,d}}{\pi^{1/2} d \Gamma (d/2)} \right)^{1/2} N^{-1/2}$.

\subsection{$L^2$ discrepancy on the complex projective space}

We can similarly define the $L^2$ discrepancy of a finite point set $A_N=\{ a_1, a_2, \ldots, a_N \} \subset \mathbb{CP}^d$ as
\[ \begin{split} &D_{\mathbb{CP}^d , 2} (A_N) \\ &= \left( \int_0^{\pi/2} \frac{1}{\mathrm{Vol}(\mathbb{CP}^d)} \int_{\mathbb{CP}^d} \left| \frac{|\{ 1 \le n \le N \, : \, a_n \in B_{\mathrm{geo}}(x,r) \} |}{N} - \frac{\mathrm{Vol}(B_{\mathrm{geo}}(x,r))}{\mathrm{Vol}(\mathbb{CP}^d)}  \right|^2 \sin (2r) \, \mathrm{d} \mathrm{Vol} (x) \, \mathrm{d}r \right)^{1/2} \\ &= \left( \int_0^1 \frac{1}{\mathrm{Vol}(\mathbb{CP}^d)} \int_{\mathbb{CP}^d} \left| \frac{|\{ 1 \le n \le N \, : \, a_n \in B_{\mathrm{FS}}(x,r) \} |}{N} - \frac{\mathrm{Vol}(B_{\mathrm{FS}}(x,r))}{\mathrm{Vol}(\mathbb{CP}^d)}  \right|^2 2 r \, \mathrm{d} \mathrm{Vol} (x) \, \mathrm{d}r \right)^{1/2}, \end{split} \]
where $B_{\mathrm{geo}}(x,r)$, resp.\ $B_{\mathrm{FS}}(x,r)$, is the open ball centered at $x \in \mathbb{CP}^d$ with radius $r$ in the geodesic metric, resp.\ Fubini--Study (chordal) metric, an analogue of the Euclidean metric on spheres. The Riesz $s$-energy $E_s (A_N)$ of a finite point set $A_N \subset \mathbb{CP}^d$ is defined in terms of the Fubini--Study metric. The Stolarsky invariance principle on $\mathbb{CP}^d$ due to Skriganov \cite[Corollary 1.1] {SKR3} states the identity
\begin{equation}\label{eq:CP Stolarsky}
N^2 D_{\mathbb{CP}^d ,2}^2 (A_N) = \frac{1}{2d+2} \left( \frac{2d}{2d+1} N^2 - E_{-1}(A_N) \right) .
\end{equation}
 
The $L^2$ discrepancy on $\mathbb{CP}^d$ was studied in \cite{SKR2, SKR3}, and the Riesz energy in \cite{ADGMS, BE2, CHS}. The optimal order of the former for any set of $N$ points is \cite[Theorem 2.2]{SKR2}
\[ \frac{1}{N^{\frac{1}{2} + \frac{1}{4d}}} \ll \inf_{|A_N|=N, \, A_N \subset \mathbb{CP}^d} D_{\mathbb{CP}^d,2}(A_N) \ll \frac{1}{N^{\frac{1}{2} + \frac{1}{4d}}}. \]

As a generalization of the spherical ensemble, Beltr\'{a}n and Etayo \cite{BE2} introduced the so-called projective ensemble on $\mathbb{CP}^d$. Given $L \in \mathbb{N}$, the projective ensemble $Z= Z(L, \mathbb{CP}^d) = \{ Z_1, Z_2, \ldots, Z_N \}$ has $N=\binom{L+d}{d}$ points and projection kernel satisfying
\begin{equation*}
|K(x,y)| = \frac{N}{\mathrm{Vol}(\mathbb{CP}^d)} \left| \left\langle \frac{x}{\|x\|}, \frac{y}{\|y\|} \right\rangle \right|^L,
\end{equation*}
with $x,y \in \mathbb{CP}^d$ being identified with corresponding points in $\mathbb{C}^{d+1}$, for ease of notation. The expected Riesz $s$-energy of the projective ensemble was shown in \cite{BE2} to satisfy for all\footnote{This was only stated for $0<s<2d$, but the proof actually works for $s<0$ as well.} $s<2d$, $s \neq 0$,
\[ \mathbb{E} E_s (Z) = \frac{d}{d-\frac{s}{2}} N^2 - d \Gamma \left( d-\frac{s}{2} \right) \frac{N^2 \Gamma (L+1)}{\Gamma \left( L+1+d - \frac{s}{2} \right)} = \frac{d}{d-\frac{s}{2}} N^2 -(1+o(1)) \frac{d \Gamma \left( d-\frac{s}{2} \right)}{(d!)^{1-\frac{s}{2d}}} N^{1+\frac{s}{2d}} . \]
The previous formula with $s=-1$ combined with the invariance principle \eqref{eq:CP Stolarsky} immediately implies that the $L^2$ discrepancy of the projective ensemble attains the optimal rate in expectation.
\begin{cor}\label{cor:capL2corollary CP} 
The projective ensemble $Z=Z(L, \mathbb{CP}^d) =\{ Z_1, Z_2, \ldots, Z_N \}$ satisfies
\[ \sqrt{\mathbb{E} D_{\mathbb{CP}^d ,2}^2 (Z)} = \sqrt{\frac{d \Gamma \left( d+\frac{1}{2} \right) \Gamma (L+1)}{(2d+2) \Gamma \left( L+1+d+\frac{1}{2} \right)}} \sim \frac{d^{\frac{1}{2}} \Gamma \left( d+\frac{1}{2} \right)^{\frac{1}{2}}}{(2d+2)^{\frac{1}{2}} (d!)^{\frac{1}{2} + \frac{1}{4d}}} \cdot \frac{1}{N^{\frac{1}{2} + \frac{1}{4d}}} . \]
\end{cor}
In the case $d=1$, $\mathbb{CP}^1$ corresponds to a 2-dimensional sphere of radius $1/2$, and the projective ensemble reduces to the spherical ensemble.

Finally, we note that the construction of the projective ensemble involved first constructing a projection kernel on $\mathbb{C}^d$ from an orthogonal system (of monomials with weight functions), then applying an inverse stereographic map from $\mathbb{C}^d$ to $\mathbb{CP}^d$ to achieve the kernel on the projective space. It is possible to construct determinantal point processes on the real, complex, quaternion, and octonian projective spaces from an orthogonal system of functions in $L^2(\mathbb{F}^d, \mathbb{C})$ in a similar manner. However, monomials are generally not orthogonal on $\mathbb{R}^d$, and are not complex-valued on  $\mathbb{H}^d$ and $\mathbb{O}^d$, so the projective ensemble cannot be generalized to the corresponding projective spaces.

\subsection{$L^2$ discrepancies on the torus}

We consider two types of $L^2$ discrepancies on the torus $\mathbb{T}^d$. The periodic $L^2$ discrepancy is defined in terms of axis-parallel boxes, whereas the ball $L^2$ discrepancy is in terms of Euclidean balls.

For $x,y \in \mathbb{T}$, the ``periodic interval'' $[x,y]_{\mathrm{per}}$ is the arc on the circle connecting the points $x$ and $y$ counterclockwise. This extends naturally to higher dimensions: for $x,y \in \mathbb{T}^d$, the ``periodic axis-parallel box'' is $[x,y]_{\mathrm{per}} = [x_1,y_1]_{\mathrm{per}} \times [x_2,y_2]_{\mathrm{per}} \times \cdots \times [x_d,y_d]_{\mathrm{per}}$. The so-called periodic $L^2$ discrepancy of a finite point set $A_N=\{ a_1, a_2, \ldots, a_N \} \subset \mathbb{T}^d$ is then defined as
\[ D_{\mathrm{per},2} (A_N) = \left( \int_{\mathbb{T}^d} \int_{\mathbb{T}^d} \left|\frac{ | \{ 1 \le n \le N \, : \, a_n \in [x,y]_{\mathrm{per}} \} |}{N} - \mathrm{Vol}([x,y]_{\mathrm{per}}) \right|^2 \, \mathrm{d} \mathrm{Vol} (x) \, \mathrm{d} \mathrm{Vol} (y) \right)^{1/2}. \]
Several variants of this notion have been considered before on the unit cube: the (ordinary) $L^2$ discrepancy is defined using only axis-parallel boxes anchored at the origin, whereas the extremal $L^2$ discrepancy uses axis-parallel boxes in the cube. For our purposes, the periodic version is the most natural, as it is compatible with the group operation on $\mathbb{T}^d$, i.e.\ it is invariant under translations of the point set $A_N$. In fact, $D_{\mathrm{per},2}$ is, up to normalizing constants, identical to the diaphony of Zinterhof \cite{Zin}. We refer to \cite{HKP} for a survey on the relationship between various versions of the $L^2$ discrepancy and the diaphony.

The smallest possible order for any set of $N$ points in $\mathbb{T}^d$ is (see \cite[Corollary 2]{HKP}, \cite[Main Theorem]{LE2})
\[ \frac{(\log N)^{\frac{d-1}{2}}}{N} \ll \inf_{|A_N|=N, \, A_N \subset \mathbb{T}^d} D_{\mathrm{per},2}(A_N) . \]
For $d=1$, the lower estimate is achieved by the $N^{\mathrm{th}}$ roots of unity, and for $d \geq 2$, the lower estimate is achieved for infinitely many $N$ \cite{HKP, KP,PI}.

As an analogue of Corollary \ref{cor:capL2corollary Sd}, we find the expected value of the periodic $L^2$ discrepancy of the harmonic ensemble on the torus. Throughout, $\gamma$ denotes Euler's constant.
\begin{thm}\label{L2torustheorem} The harmonic ensemble $X=X(T,\mathbb{T}^d)=\{ X_1, X_2, \ldots, X_N \}$ satisfies
\[ \mathbb{E} D_{\mathrm{per, 2}}^2 (X) = \frac{\log N + (\gamma+1) d}{2^{d-1} \pi^2 N^{1+\frac{1}{d}}} + O \left( \frac{(\log N)^2}{N^{1+\frac{2}{d}}} \right) \]
with an implied constant depending only on $d$. In particular,
\[ \sqrt{\mathbb{E} D_{\mathrm{per},2}^2 (X)} \sim \frac{1}{2^{\frac{d-1}{2}} \pi} \cdot \frac{(\log N)^{\frac{1}{2}}}{N^{\frac{1}{2} + \frac{1}{2d}}} . \]
\end{thm}

Let us now consider the $L^2$ discrepancy with respect to balls. The geodesic metric on the flat torus $\mathbb{T}^d$ is $\rho (x,y)=(\sum_{j=1}^d \| x_j-y_j \|^2)^{1/2}$, $x=(x_1, \ldots, x_d), y=(y_1,\ldots, y_d) \in \mathbb{T}^d$, where $\| t \|=\min_{n \in \mathbb{Z}} |t-n|$, $t \in \mathbb{R}$, denotes the distance to the nearest integer function. Let $B(x,r)$ denote the open ball in the geodesic metric centered at $x \in \mathbb{T}^d$ with radius $r>0$. We define the ball $L^2$ discrepancy of a finite point set $A_N=\{ a_1, \ldots, a_N \} \subset \mathbb{T}^d$ as
\[ D_{\mathrm{ball},2}(A_N) = \left( 2 \int_0^{1/2} \int_{\mathbb{T}^d} \left| \frac{|\{ 1 \le n \le N \, : \, a_n \in B(x,r) \}|}{N} - \mathrm{Vol}(B(x,r)) \right|^2 \, \mathrm{d}\mathrm{Vol}(x) \, \mathrm{d}r \right)^{1/2} . \]
Note that this is also invariant under translations of the point set $A_N$. The smallest possible order for any set of $N$ points in $\mathbb{T}^d$ is (see \cite{Beck3,BC,MO})
\[ \frac{1}{N^{\frac{1}{2} + \frac{1}{2d}}} \ll \inf_{|A_N|=N, \, A_N \subset \mathbb{T}^d} D_{\mathrm{ball},2}(A_N) \ll \frac{1}{N^{\frac{1}{2} + \frac{1}{2d}}} . \]

We also find the expected value of the ball $L^2$ discrepancy of the harmonic ensemble on $\mathbb{T}^d$.
\begin{thm}\label{ballL2torustheorem} The harmonic ensemble $X=X(T,\mathbb{T}^d)=\{ X_1, X_2, \ldots, X_N \}$ satisfies
\[ \mathbb{E} D_{\mathrm{ball},2}^2 (X) = \frac{\pi^{\frac{d-5}{2}}}{d 2^{d-1} \Gamma \left( \frac{d+1}{2} \right)} \cdot \frac{\log N}{N^{1+\frac{1}{d}}} + O \left( \frac{1}{N^{1+\frac{1}{d}}} \right) \]
with an implied constant depending only on $d$.
\end{thm}

In comparison, for $N$ uniformly distributed i.i.d.\ points $Y=\{ Y_1, Y_2, \ldots, Y_N \}$ on $\mathbb{T}^d$, we have $\sqrt{\mathbb{E} D_{\mathrm{per},2}^2 (Y)} = \sqrt{2^{-d}-3^{-d}} N^{- \frac{1}{2}}$ and $\sqrt{\mathbb{E} D_{\mathrm{ball},2}^2 (Y)} = c_d N^{- \frac{1}{2}}$ with some constant $c_d>0$.

The harmonic ensemble on $\mathbb{T}^d$ thus achieves the same rate in expectation in the periodic and the ball $L^2$ discrepancy, and this rate is better than that for i.i.d.\ random points. For the ball $L^2$ discrepancy, this rate is almost optimal among all point sets of the same size, as was also the case for the spherical cap $L^2$ discrepancy. For the more sensitive periodic $L^2$ discrepancy, however, the attained rate is far from optimal.

\subsection{Wasserstein metric}

Let $\rho$ denote the geodesic distance on $M$, let $\mathcal{P}(M)$ be the set of Borel probability measures on $M$, and let $1 \le p < \infty$. The $p$-Wasserstein metric is defined as
\begin{equation*} 
W_p (\mu, \nu) = \inf_{\pi \in \mathrm{Coup}(\mu, \nu)} \left( \int_{M \times M} \rho (x,y)^p \, \mathrm{d} \pi (x,y) \right)^{1/p} , \qquad \mu, \nu \in \mathcal{P}(M), 
\end{equation*}
where $\mathrm{Coup}(\mu, \nu)$ is the set of couplings (transport plans) of $\mu$ and $\nu$, that is, the set of $\pi \in \mathcal{P}(M \times M)$ such that $\pi (B \times M) = \mu (B)$ and $\pi (M \times B) = \nu(B)$ for all Borel sets $B \subseteq M$. See \cite{VI} for basic properties of $W_p$.

Approximating a given probability measure by a finitely supported one is known as the quantization problem. The best rate for the normalized Riemannian volume measure is
\begin{equation}\label{eq:quantization}
\inf_{|\mathrm{supp} (\mu) | \le N} W_p (\mu, \mathrm{Vol}/\mathrm{Vol}(M)) \sim Q(d,p) \frac{\mathrm{Vol}(M)^{1/d}}{N^{1/d}},
\end{equation}
where the infimum is over all $\mu \in \mathcal{P}(M)$ supported on at most $N$ points with arbitrary weights. The constant $Q(d,p)>0$ depends only on $p$ and the dimension $d$, but not on the geometry of $M$. For instance, $Q(2,2) = \frac{5 \sqrt{3}}{54}$ and $\sqrt{d} \ll Q(d,2) \ll \sqrt{d}$, however the precise value of $Q(d,p)$ is not known in general. We refer to \cite{GL, IA, KL} for more details.

For simplicity of notation, we will identify a finite set $A_N=\{ a_1, a_2, \ldots, a_N \} \subset M$ with its empirical measure, and write $W_p (A_N,\nu) = W_p \left( \frac{1}{N} \sum_{n=1}^N \delta_{a_n} , \nu \right)$. The main message of the following result is that the empirical measure of the harmonic and spherical ensembles on the $2$-dimensional spaces $\mathbb{T}^2$ and $\mathbb{S}^2$ attain the optimal rate $N^{-1/2}$ in \eqref{eq:quantization} in the Wasserstein metric $W_2$. In particular, the following estimates are sharp up to a constant factor.
\begin{thm}\label{thm: 2dimWassersteintheorem} \hspace{1mm}
\begin{enumerate}
\item[(i)] The harmonic ensemble $X=X(L,\mathbb{S}^2)=\{ X_1, X_2, \ldots, X_N \}$ satisfies
\[ \sqrt{\mathbb{E} W_2^2 \left( X , \mathrm{Vol} /(4 \pi) \right) } \le \frac{3.08}{\sqrt{N}} . \]
\item[(ii)] The harmonic ensemble $X=X(T,\mathbb{T}^2)=\{ X_1, X_2, \ldots, X_N \}$ satisfies
\[ \sqrt{\mathbb{E} W_2^2 \left( X, \mathrm{Vol} \right) } \le \frac{1.75}{\sqrt{N}} . \]
\item[(iii)] The spherical ensemble $Z=\{ Z_1, Z_2, \ldots, Z_N \}$ on $\mathbb{S}^2$ satisfies
\[ \sqrt{\mathbb{E} W_2^2 \left( Z, \mathrm{Vol} /(4 \pi) \right) } \le \frac{1.76}{\sqrt{N}} . \]
\end{enumerate}
\end{thm}

In comparison, $N$ uniformly distributed i.i.d.\ points $Y=\{ Y_1, Y_2, \ldots, Y_N \}$ on any $2$-dimensional compact manifold $M$ satisfy \cite{AST}
\[ \sqrt{\mathbb{E} W_2^2 \left( Y, \mathrm{Vol} / \mathrm{Vol}(M) \right) } \sim \sqrt{\frac{\mathrm{Vol}(M)}{4 \pi}} \sqrt{\frac{\log N}{N}} . \]
The appearance of the extra factor $\sqrt{\log N}$ was first observed by Ajtai, Koml\'{o}s, and Tusn\'{a}dy in connection to the optimal matching problem \cite{AKT}.

The optimal transport problem is much simpler on the one-dimensional torus $\mathbb{T}$. In this case we simply have $W_2 (A_N, \mathrm{Vol}) = 2^{-1/2} D_{\mathrm{per},2} (A_N)$ for any finite set $A_N$, see Section \ref{sec:onedimtorus}. This leads to the following explicit formula for the harmonic ensemble.
\begin{thm}\label{thm:1dimWasserstein} 
The harmonic ensemble $X=X(T,\mathbb{T})=\{ X_1, X_2, \ldots, X_N \}$ satisfies
\[ \begin{split} \mathbb{E} W_2^2 (X, \mathrm{Vol}) = \frac{1}{2} \mathbb{E} D_{\mathrm{per},2}^2 (X) &= \frac{1}{2 \pi^2 N^2} \sum_{k=1}^N \frac{1}{k} + \frac{1}{2 \pi^2 N} \sum_{k=N+1}^{\infty} \frac{1}{k^2} \\ &= \frac{\log N + \gamma +1}{2 \pi^2 N^2} + O \left( \frac{1}{N^3} \right) . \end{split} \]
In particular,
\[ \sqrt{\mathbb{E} W_2^2 (X, \mathrm{Vol})} = \frac{1}{\sqrt{2}} \sqrt{\mathbb{E} D_{\mathrm{per},2}^2 (X)} \sim \frac{\sqrt{\log N}}{\sqrt{2} \pi N}. \]
\end{thm}

In dimensions $d \ge 3$, uniformly distributed i.i.d.\ points $Y=\{ Y_1, Y_2, \ldots, Y_N \}$ in fact match the optimum in \eqref{eq:quantization}: for all large enough $N$,
\[ \sqrt{\mathbb{E} W_2^2 \left( Y , \mathrm{Vol} /\mathrm{Vol}(M) \right) } \ll \sqrt{d} \frac{\mathrm{Vol}(M)^{1/d}}{N^{1/d}} . \]
Following the proof of this fact in \cite{BO}, the same can be deduced for the harmonic ensemble on $\mathbb{S}^d$ and $\mathbb{T}^d$, $d \ge 3$. More generally, the same holds for the harmonic ensemble on a compact manifold $M$ of dimension $d \ge 3$, as defined in Remark \ref{harmoniconM}, provided that $K(x,x)$ is constant.

\section{Riesz energy on $\mathbb{S}^d$}

Given real parameters $\alpha, \beta >-1$, let $P_k^{(\alpha, \beta)} (t)$, $k \ge 0$ denote the Jacobi polynomials. They satisfy the orthogonality relation \cite[p.\ 99]{AAR}
\begin{equation}\label{orthogonality}
\int_{-1}^1 P_k^{(\alpha, \beta)} (t) P_n^{(\alpha, \beta)} (t) (1-t)^{\alpha} (1+t)^{\beta} \, \mathrm{d}t = \frac{2^{\alpha + \beta +1} \Gamma (k+\alpha +1) \Gamma (k+\beta +1)}{(2k+\alpha + \beta +1) \Gamma (k+\alpha + \beta +1) k!} \delta_{kn}.
\end{equation}
We will often use the fact that $\Gamma (k+x)/\Gamma (k+y)=(1+O(k^{-1}))k^{x-y}$, an immediate corollary of Stirling's formula.

\begin{proof}[Proof of Theorem \ref{thm:riesztheorem}]
For any $s<d$, we have the explicit formula \cite{BE2}
\begin{equation*} 
\mathbb{E} E_s (X) = \mathcal{I}_{s,d} N^2 - \frac{(d-1)! N^2}{2^{d-1+\frac{s}{2}} \Gamma \left( \frac{d}{2} \right)^2 \binom{L+\frac{d}{2}}{L}^2} \int_{-1}^1 P_L^{(\frac{d}{2}, \frac{d-2}{2})} (t)^2 (1-t)^{\frac{d-2-s}{2}} (1+t)^{\frac{d-2}{2}} \, \mathrm{d} t . 
\end{equation*}
The asymptotics in the case $0 \le s <d$ was found in \cite{BMO}, therefore we assume $s<0$. The connection formula for Jacobi polynomials \cite[p.\ 358]{AAR} yields
\begin{equation*} 
P_L^{(\frac{d}{2}, \frac{d-2}{2})}(t)= \frac{(\frac{d}{2})_L}{(d-1-\frac{s}{2})_{L+1}} \sum_{k=0}^L \left( 2k+d-1-\frac{s}{2} \right) \frac{(1+\frac{s}{2})_{L-k}(d-1-\frac{s}{2})_k (L+d)_k} {(L-k)! (\frac{d}{2})_k (L+d-\frac{s}{2})_k} P_k^{(\frac{d-2-s}{2}, \frac{d-2}{2})}(t), 
\end{equation*}
where $(y)_k=\prod_{j=0}^{k-1} (y+j)$ denotes the rising factorial. The orthogonality relation \eqref{orthogonality} thus leads to
\begin{equation*} 
\mathbb{E} E_s (X) = \mathcal{I}_{s,d} N^2 - \frac{d\cdot d! N^2 (L!)^2}{2^{2+s} \left( L+\frac{d}{2} \right)^2 \Gamma \left( L+d-\frac{s}{2} \right)^2} F(s) 
\end{equation*}
with
\begin{equation*}
 F(s)= \frac{\Gamma \left( d-1-\frac{s}{2} \right)^2}{\Gamma \left( \frac{d}{2} \right)^2} \sum_{k=0}^L \left( 2k+d-1- \frac{s}{2} \right) \frac{(1+\frac{s}{2})_{L-k}^2 (d-1-\frac{s}{2})_k^2 (L+d)_k^2 \Gamma(k+ \frac{d}{2} - \frac{s}{2}) \Gamma(k+\frac{d}{2})}{((L-k)!)^2( \frac{d}{2})_k^2 ( L + d - \frac{s}{2})_k^2 \Gamma(k+d-1-\frac{s}{2}) k!} . 
 \end{equation*}
Since $N=(1+O(L^{-1})) 2L^d /d!$, we have
\begin{equation*} 
\frac{d d! N^2 (L!)^2}{2^{2+s} \left( L+\frac{d}{2} \right)^2 \Gamma \left( L+d-\frac{s}{2} \right)^2} = \left( 1+O \left( N^{-\frac{1}{d}} \right) \right) \frac{d}{2^{s+\frac{s}{d}} (d!)^{1-\frac{s}{d}}} N^{\frac{s}{d}}, 
\end{equation*}
so it remains to find the asymptotics of $F(s)$.

Observe that
\begin{equation}\label{stirling}
\begin{split} \frac{(d-1-\frac{s}{2})_k}{(\frac{d}{2})_k} = \frac{\Gamma \left( k+d-1-\frac{s}{2} \right) \Gamma \left( \frac{d}{2} \right)}{\Gamma \left( d-1-\frac{s}{2} \right) \Gamma \left( k+\frac{d}{2} \right)} &= \left( 1+O \left( k^{-1} \right) \right) \frac{\Gamma \left( \frac{d}{2} \right)}{\Gamma \left( d-1-\frac{s}{2} \right)} k^{\frac{d}{2} -1 -\frac{s}{2}}, \\ \frac{(L+d)_k}{(L+d-\frac{s}{2})_k} = \frac{\Gamma \left( k+L+d \right) \Gamma \left( L+d-\frac{s}{2} \right)}{\Gamma \left( L+d \right) \Gamma \left( k+L+d-\frac{s}{2} \right)} &= \left( 1+O \left( L^{-1} \right) \right) (L+k)^{\frac{s}{2}} L^{-\frac{s}{2}}, \\ \frac{\Gamma \left( k+\frac{d}{2} - \frac{s}{2} \right) \Gamma \left( k+\frac{d}{2} \right)}{\Gamma \left( k+d-1-\frac{s}{2} \right) k!} &=1+O \left( k^{-1} \right) , \end{split}
\end{equation}
and also the rough estimate
\begin{equation*} 
\left| \frac{(1+\frac{s}{2})_{L-k}}{(L-k)!} \right| \ll (L-k+1)^{\frac{s}{2}} \qquad \textrm{uniformly in } 0 \le k \le L . 
\end{equation*}
Substituting \eqref{stirling} in the definition of $F(s)$ thus yields
\begin{equation}\label{Fsapproximation}
F(s)=2 \sum_{k=0}^L \frac{(1+\frac{s}{2})_{L-k}^2}{((L-k)!)^2} k^{d-1-s} (L+k)^s L^{-s} + O \left( \sum_{k=0}^L (L-k+1)^s k^{d-2-s} \right) .
\end{equation}

Assume first that $-1<s<0$. Then the error term in \eqref{Fsapproximation} is $O(L^{d-1})$. The $k=L$ term in \eqref{Fsapproximation} is $O(L^{d-1-s})$. For the terms $0 \le k <L$ we apply
\begin{equation*}
 \frac{(1+\frac{s}{2})_{L-k}}{(L-k)!} = \frac{\Gamma \left( L-k+1+\frac{s}{2} \right)}{\Gamma \left( 1+\frac{s}{2} \right) \Gamma \left( L-k+1 \right)} = \left( 1+O \left( \frac{1}{L-k} \right) \right) \frac{(L-k)^{\frac{s}{2}}}{\Gamma \left( 1+\frac{s}{2} \right)}
  \end{equation*}
to deduce
\begin{equation*}
 \begin{split} F(s) &= \frac{2}{\Gamma \left( 1+\frac{s}{2} \right)^2} \sum_{k=0}^{L-1} (L-k)^s k^{d-1-s} (L+k)^s L^{-s} + O \left( L^{d-1} + \sum_{k=0}^{L-1} (L-k)^{s-1} k^{d-1-s} \right) \\ &= \frac{2}{\Gamma \left( 1+\frac{s}{2} \right)^2} L^d \int_0^1 (1-x)^s x^{d-1-s} (1+x)^s \, \mathrm{d} x + O( L^{d-1-s}) \\ &= \frac{\Gamma \left( 1+s \right) \Gamma \left( \frac{d-s}{2} \right)}{\Gamma \left( 1+\frac{s}{2} \right)^2 \frac{d+s}{2} \Gamma \left( \frac{d+s}{2} \right)} L^d +O( L^{d-1-s}) \\ &= \left( 1+O \left( N^{-\frac{1}{d} - \frac{s}{d}} \right) \right) \frac{\Gamma \left( 1+s \right) \Gamma \left( \frac{d-s}{2} \right) d! N}{\Gamma \left( 1+\frac{s}{2} \right)^2 (d+s) \Gamma \left( \frac{d+s}{2} \right)} . \end{split} 
 \end{equation*}
The integral in the previous formula was evaluated as the value of a hypergeometric function at the point $-1$ \cite[p.\ 126]{AAR}. After some simplification involving Legendre's duplication formula for the gamma function, $\Gamma (x) \Gamma (x+1/2)=2^{1-2x} \sqrt{\pi} \Gamma (2x)$, we arrive at the claim in the case $-1<s<0$.

Assume next that $s=-1$. Then the error term in \eqref{Fsapproximation} is $O(L^{d-1} \log L)$. Let us write $F(-1)$ in the form
\begin{equation*}
 F(-1)= 2 \sum_{k=0}^L \left( \frac{(1/2)_{L-k}^2}{((L-k)!)^2} - \frac{1}{\pi (L-k+1)} \right) \frac{L k^d}{L+k} + 2 \sum_{k=0}^L \frac{L k^d}{\pi (L-k+1) (L+k)} + O(L^{d-1}\log L) . 
 \end{equation*}
Here
\[ \left| \frac{(1/2)_{L-k}^2}{((L-k)!)^2} - \frac{1}{\pi (L-k+1)} \right| \ll \frac{1}{(L-k+1)^2} , \]
and $Lk^d /(L+k) = L^d/2 + O((L-k)L^{d-1})$, hence
\begin{equation*}
 \begin{split} 2 \sum_{k=0}^L \left( \frac{(1/2)_{L-k}^2}{((L-k)!)^2} - \frac{1}{\pi (L-k+1)} \right) \frac{L k^d}{L+k} &= L^d \sum_{k=0}^L \left( \frac{(1/2)_{L-k}^2}{((L-k)!)^2} - \frac{1}{\pi (L-k+1)} \right) +O (L^{d-1} \log L) \\ &= L^d \sum_{n=0}^{\infty} \left( \frac{(1/2)_n^2}{(n!)^2} - \frac{1}{\pi (n+1)} \right) +O(L^{d-1} \log L) \\ &= \frac{4 \log 2}{\pi} L^d +O(L^{d-1} \log L) . \end{split} 
 \end{equation*}
The infinite series in the previous formula was evaluated using the asymptotics of the partial sums of zero-balanced hypergeometric series \cite[p.\ 110]{LU}. Further,
\begin{equation*}
 \begin{split} 2 \sum_{k=0}^L \frac{L k^d}{\pi (L-k+1) (L+k)} &= \frac{2L^d}{\pi} \int_0^1 \frac{x^d}{\left( 1+\frac{1}{L}-x \right) (1+x)} \, \mathrm{d}x +O( L^{d-1}) \\ &= \frac{2L^d}{\pi} \left( \int_0^1 \frac{x^d-1}{1-x^2} \, \mathrm{d}x + \int_0^1 \frac{1}{\left( 1+\frac{1}{L}-x \right) (1+x)} \, \mathrm{d}x \right) +O(L^{d-1} \log L) \\ &= \frac{1}{\pi} L^d \log L + \left( - \frac{2}{\pi} \sum_{j=1}^{\lfloor d/2 \rfloor} \frac{1}{d-2j+1} + \frac{(-1)^d}{\pi} \log 2 \right) L^d +O(L^{d-1} \log L) .  \end{split}
  \end{equation*}
Using $N=(1+O(L^{-1})) 2L^d/d!$, we thus have
\begin{equation*}
 \begin{split} F(-1) &= \frac{1}{\pi} L^d \log L + \left( \frac{4+(-1)^d}{\pi} \log 2 - \frac{2}{\pi} \sum_{j=1}^{\lfloor d/2 \rfloor} \frac{1}{d-2j+1} \right) L^d +O(L^{d-1} \log L) \\ &= \frac{(d-1)!}{2 \pi} N \log N + \frac{d!}{2} \left( \frac{\log (d! /2)}{\pi d} + \frac{4+(-1)^d}{\pi} \log 2 - \frac{2}{\pi} \sum_{j=1}^{\lfloor d/2 \rfloor} \frac{1}{d-2j+1} \right) N +O \left( N^{1-\frac{1}{d}} \log N \right) , \end{split} 
 \end{equation*}
and the claim for $s=-1$ follows.

Finally, assume that $s<-1$. Then the error term in \eqref{Fsapproximation} is $O(L^{d-2-s})$. If $s=-2$, we have $(1+\frac{s}{2})_{L-k}=0$ unless $k=L$. Otherwise, we use
\begin{equation*}
 k^{d-1-s}(L+k)^s L^{-s} = 2^s L^{d-1-s}+O((L-k)L^{d-2-s}). 
 \end{equation*}
In either case, we obtain
\begin{equation*}
 \begin{split} F(s) &= 2^{1+s} L^{d-1-s} \sum_{k=0}^L \frac{(1+\frac{s}{2})_{L-k}^2}{((L-k)!)^2} + O( L^{d-2-s} + L^d ) \\ &= 2^{1+s} L^{d-1-s} \sum_{n=0}^{\infty} \frac{(1+\frac{s}{2})_n^2}{(n!)^2} + O( L^{d-2-s}+L^d ) \\ &= 2^{1+s} L^{d-1-s} \frac{\Gamma (-1-s)}{\Gamma \left( - \frac{s}{2} \right)^2} + O( L^{d-2-s}+L^d ) \\ &=2^{s+\frac{s}{d}+\frac{1}{d}} (d!)^{1-\frac{1}{d}-\frac{s}{d}} \frac{\Gamma (-1-s)}{\Gamma \left( - \frac{s}{2} \right)^2} N^{1-\frac{1}{d} - \frac{s}{d}} +O \left( N^{1-\frac{2}{d}-\frac{s}{d}} + N \right) , \end{split}
  \end{equation*}
and the claim for $s<-1$ follows. The infinite series in the previous formula was evaluated by a formula of Gauss on the value of hypergeometric functions at the point $1$ \cite[p.\ 66]{AAR}.
\end{proof}

\section{$L^2$ discrepancies on $\mathbb{T}^d$}

\subsection{Periodic $L^2$ discrepancy}

In this section, we estimate the periodic $L^2$ discrepancy of the harmonic ensemble on $\mathbb{T}^d$, and prove Theorem \ref{L2torustheorem}. The proof is based on a Fourier analytic formula for the periodic $L^2$ discrepancy \cite{HKP,LE1}: for any finite point set $A_N=\{ a_1, a_2, \ldots, a_N \} \subset \mathbb{T}^d$, we have
\begin{equation}\label{L2discrepformula}
D_{\mathrm{per},2}^2 (A_N) = \frac{1}{3^d} \sum_{\substack{k \in \mathbb{Z}^d \\ k \neq 0}} \frac{1}{\prod_{j=1}^d r(k_j)} \left| \frac{1}{N} \sum_{n=1}^N e^{2 \pi i \langle k,a_n \rangle} \right|^2 ,
\end{equation}
where $r(k_j):=\max \{ 2 \pi^2 k_j^2 /3 ,1 \}$. For uniformly distributed i.i.d.\ points $Y=\{ Y_1, Y_2, \ldots, Y_N \}$ we have $\mathbb{E} \left| \sum_{n=1}^N e^{2 \pi i \langle k, Y_n \rangle} \right|^2 = N$, hence
\begin{equation*}
 \mathbb{E} D_{\mathrm{per},2}^2 (Y) = \frac{1}{3^d} \sum_{\substack{k \in \mathbb{Z}^d \\ k \neq 0}} \frac{1}{\prod_{j=1}^d r(k_j)} \cdot \frac{1}{N} = \frac{1}{3^d N} \left( \prod_{j=1}^d \sum_{k_j \in \mathbb{Z}} \frac{1}{r(k_j)} -1 \right) = \frac{1}{3^d N} \left( \frac{3^d}{2^d} -1 \right) . 
 \end{equation*}
In contrast, the variance of the exponential sum for the harmonic ensemble is the following.
\begin{lem}\label{harmonicvariancetoruslemma} For any $k \in \mathbb{Z}^d$, $k \neq 0$, the harmonic ensemble $X=X(T,\mathbb{T}^d)=\{ X_1, X_2, \ldots, X_N \}$ satisfies
\begin{equation*}
 \mathbb{E} \left| \sum_{n=1}^N e^{2 \pi i \langle k, X_n \rangle} \right|^2 = N - \prod_{j=1}^d \max \{ 2T+1-|k_j|, 0 \} . 
 \end{equation*}
\end{lem}

\begin{proof} By the definition of the kernel, $K(x,x)=(2T+1)^d=N$ is constant on the diagonal, and
\begin{equation*}
 |K(x,y)|^2 = \sum_{\substack{\ell, m \in \mathbb{Z}^d \\ \| \ell \|_{\infty}, \| m \|_{\infty} \le T}} e^{2 \pi i \langle \ell -m , x-y \rangle} . 
 \end{equation*}
An application of formula \eqref{varianceformula} with $f(x)=e^{2 \pi i \langle k,x \rangle}$ thus leads to
\begin{equation*}
 \begin{split} \mathbb{E} \left| \sum_{n=1}^N e^{2 \pi i \langle k, X_n \rangle} \right|^2 &= N + \int_{\mathbb{T}^d} \int_{\mathbb{T}^d} (N^2 - |K(x,y)|^2) e^{2 \pi i \langle k,x-y \rangle} \, \mathrm{d} \mathrm{Vol}(x) \, \mathrm{d} \mathrm{Vol} (y) \\ &= N - \sum_{\substack{\ell, m \in \mathbb{Z}^d \\ \| \ell \|_{\infty}, \| m \|_{\infty} \le T}} \int_{\mathbb{T}^d} \int_{\mathbb{T}^d} e^{2 \pi i \langle k+\ell -m, x-y \rangle} \, \mathrm{d} \mathrm{Vol}(x) \, \mathrm{d} \mathrm{Vol} (y) \\ &= N - \sum_{\substack{\ell, m \in \mathbb{Z}^d \\ \| \ell \|_{\infty}, \| m \|_{\infty} \le T \\ m-\ell =k}} 1 . \end{split} 
 \end{equation*}
The sum in the previous formula is simply the number of lattice points in the intersection of the cube $[-T,T]^d$ with its translate $[-T,T]^d+k$. This intersection is empty if $|k_j| \ge 2T+1$ for some $j$, and it is a lattice box of side lengths $2T-|k_j|$, $1 \le j \le d$ otherwise. The number of lattice points in the intersection is thus $\prod_{j=1}^d \max \{ 2T+1-|k_j|, 0 \}$, as claimed.
\end{proof}

\begin{proof}[Proof of Theorem \ref{L2torustheorem}] In dimension $d=1$, formula \eqref{L2discrepformula} and Lemma \ref{harmonicvariancetoruslemma} lead to the particularly simple explicit formula
\begin{equation}\label{onedimL2discrep}
\begin{split} \mathbb{E} D_{\mathrm{per},2}^2 (X) = \frac{1}{3} \sum_{\substack{k \in \mathbb{Z} \\ k \neq 0}} \frac{1}{2 \pi^2 k^2 /3} \cdot \frac{N-\max \{ N-|k|,0 \}}{N^2} &= \frac{1}{\pi^2 N^2} \sum_{k=1}^N \frac{1}{k} + \frac{1}{\pi^2 N} \sum_{k=N+1}^{\infty} \frac{1}{k^2} \\ &= \frac{\log N + \gamma +1}{\pi^2 N^2} + O \left( \frac{1}{N^3} \right) . \end{split}
\end{equation}
We now extend this to dimensions $d \ge 2$. Recall that $N=(2T+1)^d$. By formula \eqref{L2discrepformula} and Lemma \ref{harmonicvariancetoruslemma},
\begin{equation}\label{highdimL2discrep}
\mathbb{E} D_{\mathrm{per},2}^2 (X) = \frac{1}{3^d} \sum_{\substack{k \in \mathbb{Z}^d \\ k \neq 0}} \frac{1}{\prod_{j=1}^d r(k_j)} \cdot \frac{N-\prod_{j=1}^d \max \{ 2T+1-|k_j| ,0 \}}{N^2} .
\end{equation}
Consider first the terms $k \in [-N^{\frac{1}{d}},N^{\frac{1}{d}}]^d$, $k \neq 0$. Expanding the product leads to
\begin{equation*}
 \prod_{j=1}^d (2T+1-|k_j|) = (2T+1)^d - (2T+1)^{d-1} \sum_{j=1}^d |k_j| + (2T+1)^{d-2} \sum_{\substack{j_1, j_2 =1 \\ j_1 \neq j_2}}^d |k_{j_1}| \cdot |k_{j_2}| + \cdots , 
 \end{equation*}\
 therefore
\begin{equation}\label{variancefactor}
N-\prod_{j=1}^d (2T+1-|k_j|) = N^{1-\frac{1}{d}} \sum_{j=1}^d |k_j| - N^{1-\frac{2}{d}} \sum_{\substack{j_1, j_2 =1 \\ j_1 \neq j_2}}^d |k_{j_1}| \cdot |k_{j_2}| + \cdots .
\end{equation}
The cancellation of the main term $N$ is the manifestation of the repulsive nature of the system in Fourier space. The main contribution comes from the first term in \eqref{variancefactor}:
\begin{equation*}
 \begin{split} \frac{1}{3^d} \sum_{\substack{k \in [-N^{\frac{1}{d}},N^{\frac{1}{d}}]^d \\ k \neq 0}} \frac{1}{\prod_{j=1}^d r(k_j)} \cdot \frac{\sum_{j=1}^d |k_j|}{N^{1+\frac{1}{d}}} &= \frac{d}{3^d N^{1+\frac{1}{d}}} \sum_{\substack{k \in [-N^{\frac{1}{d}}, N^{\frac{1}{d}}]^d \\ k_1 \neq 0}} \frac{|k_1|}{\prod_{j=1}^d r(k_j)} \\ &= \frac{d}{3^d N^{1+\frac{1}{d}}} \bigg( \sum_{\substack{-N^{\frac{1}{d}} \le k_1 \le N^{\frac{1}{d}} \\ k_1 \neq 0}} \frac{3}{2 \pi^2 |k_1|} \bigg) \prod_{j=2}^{d} \sum_{-N^{\frac{1}{d}} \le k_j \le N^{\frac{1}{d}}} \frac{1}{r(k_j)} \\ &= \frac{d}{3^d N^{1+\frac{1}{d}}} \left( \frac{3 \log N^{\frac{1}{d}} + 3 \gamma}{\pi^2} + O \left( \frac{1}{N^{\frac{1}{d}}} \right) \right) \left( \frac{3}{2} + O \left( \frac{1}{N^{\frac{1}{d}}} \right) \right)^{d-1} \\ &= \frac{\log N + \gamma d}{2^{d-1} \pi^2 N^{1+\frac{1}{d}}} + O \left( \frac{\log N}{N^{1+\frac{2}{d}}} \right) . \end{split} 
 \end{equation*}
 The contribution of the second term in \eqref{variancefactor} is negligible:
\begin{equation*}
 \frac{1}{3^d} \sum_{\substack{k \in [-N^{\frac{1}{d}},N^{\frac{1}{d}}]^d \\ k \neq 0}} \frac{1}{\prod_{j=1}^d r(k_j)} \cdot \frac{\sum_{j_1 \neq j_2} |k_{j_1}| \cdot |k_{j_2}|}{N^{1+\frac{2}{d}}} = \frac{\binom{d}{2}}{3^d N^{1+\frac{2}{d}}} \sum_{\substack{k \in [-N^{\frac{1}{d}},N^{\frac{1}{d}}]^d \\ k \neq 0}} \frac{|k_1| \cdot |k_2|}{\prod_{j=1}^d r(k_j)} \ll \frac{(\log N)^2}{N^{1+\frac{2}{d}}} . 
 \end{equation*}
All other terms in \eqref{variancefactor} are easily seen to have an even smaller contribution. The total contribution of all $k \in [-N^{\frac{1}{d}},N^{\frac{1}{d}}]^d$, $k \neq 0$ in \eqref{highdimL2discrep} is thus
\begin{equation*}
 \frac{1}{3^d} \sum_{\substack{k \in [-N^{\frac{1}{d}}, N^{\frac{1}{d}}]^d \\ k \neq 0}} \frac{1}{\prod_{j=1}^d r(k_j)} \cdot \frac{N-\prod_{j=1}^d \max \{ 2T+1-|k_j| ,0 \}}{N^2} = \frac{\log N + \gamma d}{2^{d-1} \pi^2 N^{1+\frac{1}{d}}} + O \left( \frac{(\log N)^2}{N^{1+\frac{2}{d}}} \right) . 
 \end{equation*}

Note that for all other terms $k \in \mathbb{Z}^d \backslash [-N^{\frac{1}{d}}, N^{\frac{1}{d}}]^d$, we have $N-\prod_{j=1}^d \max \{ 2T+1-|k_j| ,0 \} =N$. The contribution of all terms with $|k_1|>N^{1/d}$ and $|k_2|, \ldots, |k_d| \le N^{\frac{1}{d}}$ in \eqref{highdimL2discrep} is
\begin{equation*}
 \begin{split} \frac{1}{3^d} \sum_{\substack{|k_1|>N^{1/d} \\ |k_2|, \ldots, |k_d| \le N^{\frac{1}{d}}}} \frac{1/N}{\prod_{j=1}^d r(k_j)} &= \frac{1}{3^d N} \sum_{|k_1| >N^{1/d}} \frac{3}{2 \pi^2 k_1^2} \prod_{j=2}^d \sum_{|k_j| \le N^{1/d}} \frac{1}{r(k_j)} \\ &= \frac{1}{3^d N} \left( \frac{3}{\pi^2 N^{\frac{1}{d}}} + O \left( \frac{1}{N^{\frac{2}{d}}} \right) \right) \left( \frac{3}{2} + O \left( \frac{1}{N^{\frac{1}{d}}} \right) \right)^{d-1} \\ &= \frac{1}{2^{d-1} \pi^2 N^{1+\frac{1}{d}}} + O \left( \frac{1}{N^{1+\frac{2}{d}}} \right) . \end{split} 
 \end{equation*}
For any given $1 \le j \le d$, the same holds for the terms with $|k_j|>N^{\frac{1}{d}}$, $|k_{j'}| \le N^{\frac{1}{d}}$, $j' \neq j$. The contribution of all terms $k \in \mathbb{Z}^d \backslash [-N^{\frac{1}{d}}, N^{\frac{1}{d}}]^d$ with at least two coordinates being $>N^{\frac{1}{d}}$ in absolute value, is similarly seen to be $O(N^{-1-\frac{2}{d}})$. The total contribution of all $k \in \mathbb{Z}^d \backslash [-N^{\frac{1}{d}}, N^{\frac{1}{d}}]^d$ in \eqref{highdimL2discrep} is thus
\begin{equation*}
 \frac{1}{3^d} \sum_{k \in \mathbb{Z}^d \backslash [-N^{\frac{1}{d}}, N^{\frac{1}{d}}]^d} \frac{1}{\prod_{j=1}^d r(k_j)} \cdot \frac{N-\prod_{j=1}^d \max \{ 2T+1-|k_j| ,0 \}}{N^2} = \frac{d}{2^{d-1} \pi^2 N^{1+\frac{1}{d}}} + O \left( \frac{1}{N^{1+\frac{2}{d}}} \right) . 
 \end{equation*}
Formula \eqref{highdimL2discrep} finally shows that
\begin{equation*}
 \mathbb{E} D_{\mathrm{per},2}^2 (X) = \frac{\log N + (\gamma +1)d}{2^{d-1} \pi^2 N^{1+\frac{1}{d}}} + O \left( \frac{(\log N)^2}{N^{1+\frac{2}{d}}} \right) , 
 \end{equation*}
as claimed.
\end{proof}

\subsection{Ball $L^2$ discrepancy}

In this section, we estimate the ball $L^2$ discrepancy of the harmonic ensemble on $\mathbb{T}^d$, and prove Theorem \ref{ballL2torustheorem}. Our starting point is a Fourier analytic formula for the ball $L^2$ discrepancy, similar to the diaphony \eqref{L2discrepformula}.
\begin{lem}\label{balldiaphonylemma} For any finite point set $A_N = \{ a_1, \ldots, a_N \} \subset \mathbb{T}^d$, we have
\[ D_{\mathrm{ball},2}^2 (A_N) = \sum_{\substack{k \in \mathbb{Z}^d \\ k \neq 0}} b_k \left|\frac{1}{N} \sum_{n=1}^N e^{2 \pi i \langle k,a_n \rangle} \right|^2 \]
with some positive real numbers $b_k=1/(d 2^d \pi^2 |k|^{d+1}) + O(1/|k|^{d+2})$.
\end{lem}

\begin{proof} Using the formula for the Fourier transform of the indicator function of the Euclidean unit ball \cite[p.\ 155]{SW}, we immediately deduce that for any $0< r \le 1/2$, the Fourier coefficients of the indicator function $I_{B(0,r)}$ of the ball $B(0,r) \subset \mathbb{T}^d$ are
\[ \int_{\mathbb{T}^d} I_{B(0,r)} (x) e^{-2 \pi i \langle k,x \rangle} \, \mathrm{d}\mathrm{Vol}(x) = r^{d/2} |k|^{-d/2} J_{d/2}(2 \pi r |k|), \qquad k \in \mathbb{Z}^d, k \neq 0, \]
where $J_{\alpha}$ are the Bessel functions of the first kind. Since $I_{B(x,r)}(y) = I_{B(0,r)}(x-y)$, for any fixed $y \in \mathbb{T}^d$ we obtain the Fourier series expansion in the variable $x$
\[ I_{B(x,r)}(y) = \mathrm{Vol} (B(0,r)) + \sum_{\substack{k \in \mathbb{Z}^d \\ k \neq 0}} r^{d/2} |k|^{-d/2} J_{d/2}(2 \pi r |k|) e^{2 \pi i \langle k,x-y \rangle} . \]
Setting $y=a_n$ and then summing over $1 \le n \le N$ leads to the Fourier series expansion
\begin{multline*}
\frac{|\{ 1 \le n \le N \, : \, a_n \in B(x,r) \}|}{N} - \mathrm{Vol}(B(x,r)) \\ = \sum_{\substack{k \in \mathbb{Z}^d \\ k \neq 0}} r^{d/2} |k|^{-d/2} J_{d/2}(2 \pi r |k|) \left( \frac{1}{N} \sum_{n=1}^N e^{-2 \pi i \langle k,a_n \rangle} \right) e^{2 \pi i \langle k, x \rangle} .
\end{multline*}
An application of the Parseval formula thus yields the desired formula with
\[ b_k = 2 \int_0^{1/2} r^d |k|^{-d} J_{d/2}^2 (2 \pi r |k|) \, \mathrm{d}r = \frac{1}{2^d \pi^{d+1} |k|^{2d+1}} \int_0^{\pi |k|} x^d J_{d/2}^2 (x) \, \mathrm{d}x . \]
The asymptotics of the Bessel functions \cite[p.\ 209]{AAR} shows that
\[ x^d J_{d/2}^2 (x) = \frac{2x^{d-1}}{\pi} \cos^2 \left( x-\frac{(d+1)\pi}{4} \right) + O(x^{d-2}) \qquad \textrm{as } x \to \infty. \]
Hence
\[  \int_0^{\pi |k|} x^d J_{d/2}^2 (x) \, \mathrm{d}x = \frac{\pi^{d-1} |k|^d}{d} + O (|k|^{d-1}) , \]
and the claim follows.
\end{proof}

For uniformly distributed i.i.d.\ points $Y=\{ Y_1, Y_2, \ldots, Y_N \}$ we have $\mathbb{E} \left| \sum_{n=1}^N e^{2 \pi i \langle k, Y_n \rangle} \right|^2 = N$, hence $\sqrt{\mathbb{E} D_{\mathrm{ball},2}^2 (Y)} = c_d N^{-1/2}$ with the constant $c_d=\left( \sum_{k \in \mathbb{Z}^d, \, k \neq 0} b_k \right)^{1/2}$.

\begin{proof}[Proof of Theorem \ref{ballL2torustheorem}] Recall that $N=(2T+1)^d$. Lemmas \ref{harmonicvariancetoruslemma} and \ref{balldiaphonylemma} immediately yield
\begin{equation}\label{ballL2explicit}
\mathbb{E} D_{\mathrm{ball},2}^2 (X) = \sum_{\substack{k \in \mathbb{Z}^d \\ k \neq 0}} b_k \frac{N-\prod_{j=1}^d \max\{ 2T+1-|k_j|,0 \}}{N^2}
\end{equation}
with some positive real numbers $b_k=1/(d 2^d \pi^2 |k|^{d+1}) + O(1/|k|^{d+2})$. First note that the contribution of the terms $|k|>N^{1/d}$ in \eqref{ballL2explicit} is negligible:
\begin{equation}\label{ballestimate1}
\sum_{|k|>N^{1/d}} b_k \frac{N-\prod_{j=1}^d \max\{ 2T+1-|k_j|,0 \}}{N^2} \ll \sum_{|k|>N^{1/d}} \frac{1}{|k|^{d+1}N} \ll \frac{1}{N^{1+\frac{1}{d}}} .
\end{equation}

Consider now the terms $0<|k|\le N^{1/d}$ in \eqref{ballL2explicit}. For all such terms, $|k_j| \le 2T+1$ for all $1 \le j \le d$, and formula \eqref{variancefactor} holds. The main contribution comes from the first term on the right hand side of \eqref{variancefactor}:
\begin{equation}\label{ballestimate2}
\begin{split} \sum_{0<|k| \le N^{1/d}} b_k \frac{\sum_{j=1}^d |k_j|}{N^{1+\frac{1}{d}}} &= \sum_{0<|k| \le N^{1/d}} \frac{\sum_{j=1}^d |k_j|}{d 2^d \pi^2 |k|^{d+1} N^{1+\frac{1}{d}}} + O \left( \sum_{0<|k| \le N^{1/d}} \frac{1}{|k|^{d+2}} \cdot \frac{|k|}{N^{1+\frac{1}{d}}} \right) \\ &= \frac{1}{2^d \pi^2 N^{1+\frac{1}{d}}} \sum_{0<|k|\le N^{1/d}} \frac{|k_1|}{|k|^{d+1}} + O \left( \frac{1}{N^{1+\frac{1}{d}}} \right) . \end{split}
\end{equation}
The previous sum is easily estimated in dimension $d=1$. We now show how to estimate it in dimensions $d \ge 2$. Fix $k_1 \neq 0$, and let us estimate the inner sum over $k_2, \ldots, k_d$ with the corresponding integral. Let $f(x)=|k_1|/(k_1^2+|x|^2)^{(d+1)/2}$, $x \in \mathbb{R}^{d-1}$. The gradient has norm $|\nabla f(x)| = (d+1) |k_1| \cdot |x|/(k_1^2 + |x|^2)^{(d+3)/2}$, therefore
\[ |f(x) - f(k_2, \ldots, k_d)| \ll \frac{|k_1| \cdot |x|}{(k_1^2 + |x|^2)^{(d+3)/2}} \quad \textrm{for } x - (k_2, \ldots, k_d) \in \left[ - \frac{1}{2}, \frac{1}{2} \right]^{d-1} . \]
Note that
\[ \int_{\mathbb{R}^{d-1}} \frac{|k_1| \cdot |x|}{(k_1^2 + |x|^2)^{(d+3)/2}} \, \mathrm{d}x \ll \int_0^{\infty} \frac{|k_1|R}{(k_1^2+R^2)^{(d+3)/2}} R^{d-2} \, \mathrm{d}R \ll \frac{1}{k_1^2} , \]
and that
\[ \int_{\{ x \in \mathbb{R}^{d-1} \, : \, \left| |x| - \sqrt{N^{2/d} - k_1^2} \right| \ll 1 \}} \frac{|k_1|}{(k_1^2+|x|^2)^{(d+1)/2}} \, \mathrm{d} x \ll \frac{|k_1|}{N^{3/d}} \ll \frac{1}{k_1^2} . \]
Integrating over the unit cube centered at the lattice point $(k_2, \ldots, k_d)$ and then summing over $(k_2, \ldots, k_d)$ thus leads to
\[ \begin{split} \sum_{\substack{k_2, \ldots, k_d \in \mathbb{Z} \\ k_2^2+\cdots +k_d^2 \le N^{2/d}-k_1^2}} &\frac{|k_1|}{(k_1^2 + k_2^2 + \cdots +k_d^2)^{(d+1)/2}} \\ &= \int_{\{ x \in \mathbb{R}^{d-1} \, : \, |x| \le \sqrt{N^{2/d} - k_1^2} \}} \frac{|k_1|}{(k_1^2+|x|^2)^{(d+1)/2}} \, \mathrm{d} x + O \left( \frac{1}{k_1^2} \right) \\ &= \int_0^{\sqrt{N^{2/d}-k_1^2}} \frac{|k_1|}{(k_1^2+R^2)^{(d+1)/2}} \mathrm{Vol} (\mathbb{S}^{d-2}) R^{d-2} \, \mathrm{d}R + O \left( \frac{1}{k_1^2} \right) \\ &= \frac{\mathrm{Vol} (\mathbb{S}^{d-2})}{|k_1|} \int_0^{\sqrt{N^{2/d}/k_1^2 -1}} \frac{R^{d-2}}{(1+R^2)^{(d+1)/2}} \, \mathrm{d}R + O \left( \frac{1}{k_1^2} \right) . \end{split} \]
Replacing the upper limit of integration in the previous integral by infinity yields
\[ \int_0^{\sqrt{N^{2/d}/k_1^2 -1}} \frac{R^{d-2}}{(1+R^2)^{(d+1)/2}} \, \mathrm{d}R = \int_0^{\infty} \frac{R^{d-2}}{(1+R^2)^{(d+1)/2}} \, \mathrm{d}R + O \left( \frac{k_1^2}{N^{2/d}} \right) = \frac{1}{d-1} +  O \left( \frac{k_1^2}{N^{2/d}} \right) . \]
Summing over $0<|k_1| \le N^{1/d}$ thus leads to
\[ \begin{split} \sum_{0<|k|\le N^{1/d}} \frac{|k_1|}{|k|^{d+1}} &= \frac{\mathrm{Vol}(\mathbb{S}^{d-2})}{d-1} \sum_{0<|k_1| \le N^{1/d}} \frac{1}{|k_1|} + O \left( \sum_{0<|k_1|\le N^{1/d}} \frac{1}{k_1^2} + \sum_{0<|k_1| \le N^{1/d}} \frac{|k_1|}{N^{2/d}} \right) \\ &= \frac{2 \pi^{\frac{d-1}{2}}}{d \Gamma (\frac{d+1}{2})} \log N + O(1) . \end{split} \]
Note that the previous formula holds in dimension $d=1$ as well. Formula \eqref{ballestimate2} thus simplifies to
\[ \sum_{0<|k| \le N^{1/d}} b_k \frac{\sum_{j=1}^d |k_j|}{N^{1+\frac{1}{d}}} = \frac{\pi^{\frac{d-5}{2}}}{d 2^{d-1} \Gamma \left( \frac{d+1}{2} \right)} \cdot \frac{\log N}{N^{1+\frac{1}{d}}} + O \left( \frac{1}{N^{1+\frac{1}{d}}} \right) . \]

The second term on the right hand side of \eqref{variancefactor} is negligible:
\[ \sum_{0<|k| \le N^{1/d}} b_k \frac{\sum_{j_1 \neq j_2} |k_{j_1}| \cdot |k_{j_2}|}{N^{1+\frac{2}{d}}} \ll \sum_{0<|k| \le N^{1/d}} \frac{1}{|k|^{d-1}N^{1+\frac{2}{d}}} \ll \frac{1}{N^{1+\frac{1}{d}}} . \]
A similar argument shows that the contribution of all later terms on the right hand side of \eqref{variancefactor} is also $\ll 1/N^{1+\frac{1}{d}}$. Hence
\[ \sum_{0<|k| \le N^{1/d}} b_k \frac{N-\prod_{j=1}^d \max\{ 2T+1-|k_j|,0 \}}{N^2} = \frac{\pi^{\frac{d-5}{2}}}{d 2^{d-1} \Gamma \left( \frac{d+1}{2} \right)} \cdot \frac{\log N}{N^{1+\frac{1}{d}}} + O \left( \frac{1}{N^{1+\frac{1}{d}}} \right) , \]
which together with \eqref{ballL2explicit} and \eqref{ballestimate1} prove the claim.
\end{proof}

\section{Wasserstein metric}\label{sec:Wasserstein}

The main tool in estimating the harmonic and spherical ensembles in the quadratic Wasserstein metric is a smoothing inequality for $W_2$ on compact manifolds due to the first author \cite{BO}. In the special case of the sphere, it states that for any finite set $A_N=\{ a_1, a_2, \ldots, a_N \} \subset \mathbb{S}^2$ and any real $t>0$,
\begin{equation}\label{spheresmoothing}
W_2 (A_N, \mathrm{Vol}/(4 \pi)) \le (2t)^{1/2} + 2 \left( \sum_{\ell =1}^{\infty} \frac{e^{-\ell (\ell +1) t}}{\ell (\ell +1)} \sum_{m=-\ell}^{\ell} \left| \frac{1}{N} \sum_{n=1}^N Y_{\ell}^m (a_n) \right|^2 \right)^{1/2} .
\end{equation}
Similarly, for any finite set $A_N=\{ a_1, a_2, \ldots, a_N \} \subset \mathbb{T}^2$ and any real $t>0$,
\begin{equation}\label{torussmoothing}
W_2 (A_N,\mathrm{Vol}) \le (2t)^{1/2} + 2 \left( \sum_{\substack{k \in \mathbb{Z}^2 \\ k \neq 0}} \frac{e^{-4 \pi^2 |k|^2 t}}{4 \pi^2 |k|^2} \left| \frac{1}{N} \sum_{n=1}^N e^{2 \pi i \langle k,a_n \rangle} \right|^2 \right)^{1/2} .
\end{equation}

Let $P_{\ell}(t)$, $\ell \ge 0$ denote the Legendre polynomials with the standard normalization $P_{\ell}(1)=1$. They are the special case of the Jacobi polynomials $P_{\ell}^{(\alpha, \beta)}(t)$ with $\alpha=\beta=0$, and by \eqref{orthogonality} satisfy the orthogonality relation
\begin{equation*}
 \int_{-1}^1 P_{\ell}(t) P_k (t) \, \mathrm{d}t = \frac{2}{2\ell +1} \delta_{\ell k} . 
 \end{equation*}
Most importantly, Legendre polynomials appear in the addition formula for spherical harmonics: for any $\ell \ge 0$ and any $x,y \in \mathbb{S}^2$,
\begin{equation}\label{addition}
\sum_{m=-\ell}^{\ell} Y_{\ell}^m (x) \overline{Y_{\ell}^m(y)} = \frac{2\ell +1}{4 \pi} P_{\ell} (\langle x,y \rangle) .
\end{equation}

\subsection{Harmonic ensemble on $\mathbb{S}^2$}

Consider the harmonic ensemble $X=X(L, \mathbb{S}^2)=\{ X_1, X_2, \ldots, X_N \}$. The number of points is now $N=(L+1)^2$. We start with a variance estimate, and then give the proof of Theorem \ref{thm: 2dimWassersteintheorem} (i).
\begin{lem}\label{harmonicvariancespherelemma} For any $\ell >0$,
\begin{equation*}
 \sum_{m=-\ell}^{\ell} \mathbb{E} \left| \sum_{n=1}^N Y_{\ell}^m (X_n) \right|^2 \le \frac{2^{3/2} \cdot 3}{\pi^2} \ell (\ell +1) N^{1/2} . 
 \end{equation*}
\end{lem}

\begin{proof} By the addition formula \eqref{addition}, the kernel is $K(x,y) = \sum_{\ell =0}^L \frac{2 \ell +1}{4 \pi} P_{\ell} (\langle x,y \rangle)$. In particular, it is constant on the diagonal: $K(x,x) = N/(4 \pi)$. An application of \eqref{varianceformula} with $f(x)=Y_{\ell}^m (x)$ thus gives
\begin{equation*}
 \mathbb{E} \left| \sum_{n=1}^N Y_{\ell}^m (X_n) \right|^2 = \frac{N}{4 \pi} - \int_{\mathbb{S}^2} \int_{\mathbb{S}^2} |K(x,y)|^2 Y_{\ell}^m (x) \overline{Y_{\ell}^m (y)} \, \mathrm{d} \mathrm{Vol} (x) \, \mathrm{d} \mathrm{Vol} (y) ,  
 \end{equation*}
and the addition formula \eqref{addition} yields the explicit formula
\begin{equation}\label{explicitvariance1}
\begin{split}  \sum_{m=-\ell}^{\ell} &\mathbb{E} \left| \sum_{n=1}^N Y_{\ell}^m (X_n) \right|^2 \\ &= \frac{(2 \ell +1)N}{4 \pi} - \int_{\mathbb{S}^2} \int_{\mathbb{S}^2} |K(x,y)|^2 \frac{2 \ell +1}{4 \pi} P_{\ell} (\langle x,y \rangle) \, \mathrm{d} \mathrm{Vol} (x) \, \mathrm{d} \mathrm{Vol} (y) \\ &= \frac{2 \ell +1}{4 \pi} \left( N- \int_{\mathbb{S}^2} \int_{\mathbb{S}^2} \sum_{\ell_1, \ell_2 =0}^{L} \frac{(2 \ell_1 +1)(2 \ell_2 +1)}{16 \pi^2} P_{\ell_1} (\langle x,y \rangle) P_{\ell_2} (\langle x,y \rangle)  P_{\ell}  (\langle x,y \rangle) \, \mathrm{d} \mathrm{Vol} (x) \, \mathrm{d} \mathrm{Vol} (y) \right)  \\ &= \frac{2 \ell +1}{4 \pi} \left( N- \sum_{\ell_1, \ell_2 =0}^L \frac{(2 \ell_1+1)(2 \ell_2+1)}{2} \int_{-1}^1 P_{\ell_1} (t) P_{\ell_2}(t) P_{\ell}(t) \, \mathrm{d} t \right) . \end{split}
\end{equation}

The fact that the Legendre polynomials form an orthogonal basis in $L^2 ([-1,1])$ means in particular, that the polynomial $P_{\ell_1}(t) P_{\ell}(t)$ can be expanded as a finite sum $P_{\ell_1}(t) P_{\ell}(t) = \sum_{j =0}^{\ell_1 + \ell} a_{j} P_{j} (t)$ with coefficients $a_{j}=\frac{2 j +1}{2} \int_{-1}^1 P_{\ell_1}(t) P_{\ell}(t) P_{j} (t) \, \mathrm{d} t$.  Therefore
\begin{equation*}
 \sum_{\ell_2=0}^{\infty} \frac{2 \ell_2+1}{2} \int_{-1}^{1} P_{\ell_1} (t) P_{\ell_2}(t) P_{\ell}(t) \, \mathrm{d} t = P_{\ell_1} (1) P_{\ell}(1) =1 , 
 \end{equation*}
and consequently
\begin{equation*}
 \sum_{\ell_1=0}^L \sum_{\ell_2=0}^{\infty} \frac{(2 \ell_1 +1)(2 \ell_2 +1)}{2} \int_{-1}^1 P_{\ell_1} (t) P_{\ell_2}(t) P_{\ell}(t) \, \mathrm{d} t = \sum_{\ell_1=0}^L (2 \ell_1 +1) =N . 
 \end{equation*}
As a manifestation of the repulsive nature of the system in Fourier space, the main term $N$ in the variance cancels, and we obtain
\begin{equation}\label{explicitvariance2}
\sum_{m=-\ell}^{\ell} \mathbb{E} \left| \sum_{n=1}^N Y_{\ell}^m (X_n) \right|^2 = \frac{2 \ell +1}{4 \pi} \sum_{\ell_1=0}^L \sum_{\ell_2=L+1}^{\infty} \frac{(2 \ell_1 +1)(2 \ell_2 +1)}{2} \int_{-1}^1 P_{\ell_1} (t) P_{\ell_2}(t) P_{\ell}(t) \, \mathrm{d} t .
\end{equation}
Set $g=(\ell_1+\ell_2+\ell)/2$. The remaining integral vanishes unless $g \ge \ell_1, \ell_2, \ell$, as can be seen from comparing the degree of the product of two polynomials to the third one. Otherwise, we have the explicit formula \cite[p.\ 321]{AAR}
\begin{equation*}
 \frac{1}{2} \int_{-1}^1 P_{\ell_1} (t) P_{\ell_2}(t) P_{\ell}(t) \, \mathrm{d} t = \left\{ \begin{array}{ll} \frac{1}{\ell_1+\ell_2+\ell+1} \cdot \frac{\binom{2g-2\ell_1}{g-\ell_1} \binom{2g-2\ell_2}{g-\ell_2} \binom{2g-2\ell}{g-\ell}}{\binom{2g}{g}} & \textrm{if } g \in \mathbb{Z}, \\ 0 & \textrm{if } g \not\in \mathbb{Z} . \end{array} \right. 
 \end{equation*}

Assume first that $\ell \le L/2$. The central binomial coefficients satisfy the general inequalities $\frac{4^n}{\sqrt{\pi (n+1/2)}} \le \binom{2n}{n} \le \frac{4^n}{\sqrt{\pi (n+1/4)}}$ for all $n \ge 0$. Therefore
\begin{equation*}
 \frac{\binom{2g-2\ell_1}{g-\ell_1} \binom{2g-2\ell_2}{g-\ell_2} \binom{2g-2\ell}{g-\ell}}{\binom{2g}{g}} \le \frac{\frac{4^{g-\ell_1}}{\sqrt{\pi (g-\ell_1)}} \cdot \frac{4^{g-\ell_2}}{\sqrt{\pi (g-\ell_2+1/4)}} \cdot \frac{4^{g-\ell}}{\sqrt{\pi (g-\ell )}}}{\frac{4^g}{\sqrt{\pi (g+1/2)}}} = \frac{1}{\pi} \sqrt{\frac{g+1/2}{(g-\ell_1)(g-\ell_2+1/4)(g-\ell)}} , 
 \end{equation*}
which in turn leads to
\begin{equation*}
 \frac{(2 \ell_1 +1) (2 \ell_2+1)}{2} \int_{-1}^1 P_{\ell_1} (t) P_{\ell_2}(t) P_{\ell}(t) \, \mathrm{d} t \le \frac{(2 \ell_1 +1)(2 \ell_2 +1)}{\sqrt{\ell_1+\ell_2+\ell +1}} \cdot \frac{1}{\pi \sqrt{2 (g-\ell_1)(g-\ell_2+1/4)(g-\ell)}} . 
 \end{equation*}
By the assumptions $\ell \le L/2$ and $\ell_2 \ge L+1$, here
\begin{equation*}
 \frac{\sqrt{2\ell_2 +1}}{\sqrt{\ell_1 + \ell_2 + \ell +1}} \le \sqrt{2} \qquad \textrm{and} \qquad \frac{\sqrt{2 \ell_2 +1}}{\sqrt{g-\ell}} \le 2 \sqrt{2} , 
 \end{equation*}
hence
\begin{equation*}
 \frac{(2 \ell_1 +1) (2 \ell_2+1)}{2} \int_{-1}^1 P_{\ell_1} (t) P_{\ell_2}(t) P_{\ell}(t) \, \mathrm{d} t \le \frac{4(2L+1)}{\pi \sqrt{2 (g-\ell_1)(g-\ell_2+1/4)}} . 
 \end{equation*}
Recall that in \eqref{explicitvariance2} we are summing over $0 \le \ell_1 \le L$ and $\ell_2 \ge L+1$, and that for all nonvanishing terms, $\ell_2-\ell_1 \le \ell$. The number of terms for which $\ell_2 - \ell_1 =j$ is exactly $j$, and in this case $g-\ell_1 = (\ell +j)/2$ and $g-\ell_2+1/4 =(\ell -j)/2+1/4$. Therefore
\begin{equation*}
 \sum_{\ell_1=0}^L \sum_{\ell_2=L+1}^{\infty} \frac{(2 \ell_1 +1)(2 \ell_2 +1)}{2} \int_{-1}^1 P_{\ell_1} (t) P_{\ell_2}(t) P_{\ell}(t) \, \mathrm{d} t \le \frac{4 \sqrt{2} (2L+1)}{\pi} \sum_{j=1}^{\ell} \frac{j}{\sqrt{(\ell +j)(\ell -j +1/2)}} . 
 \end{equation*}
Here
\begin{equation*}
  \sum_{j=1}^{\ell} \frac{j}{\sqrt{(\ell +j)(\ell -j +1/2)}} \le \sqrt{\ell} +  \sum_{j=1}^{\ell -1} \frac{j}{\sqrt{\ell^2-j^2}} \le \sqrt{\ell} + \ell \int_0^1 \frac{x}{\sqrt{1-x^2}} \, \mathrm{d} x = \sqrt{\ell} + \ell ,
   \end{equation*}
and \eqref{explicitvariance2} finally yields
\begin{equation*}
 \sum_{m=-\ell}^{\ell} \mathbb{E} \left| \sum_{n=1}^N Y_{\ell}^m (X_n) \right|^2 \le \frac{2\ell +1}{4 \pi} \cdot \frac{4 \sqrt{2}(2L+1)}{\pi} ( \sqrt{\ell} + \ell ) .  
 \end{equation*}
Observe that $(2 \ell +1) (\sqrt{\ell} + \ell ) \le 3 \ell (\ell +1)$, and $2L+1 \le 2 N^{1/2}$. This finishes the proof in the case $\ell \le L/2$.

If $\ell > L/2$, the proof is much simpler. In this case formula \eqref{explicitvariance1} yields the trivial upper bound (cf.\ uniformly distributed i.i.d.\ points)
\begin{equation*}
 \sum_{m=-\ell}^{\ell} \mathbb{E} \left| \sum_{n=1}^N Y_{\ell}^m (X_n) \right|^2 \le \frac{(2 \ell +1) N}{4 \pi} , 
 \end{equation*}
which suffices for the claim of the lemma.
\end{proof}

\begin{proof}[Proof of Theorem \ref{thm: 2dimWassersteintheorem} (i)] An application of the smoothing inequality \eqref{spheresmoothing}, the triangle inequality for the $L^2$ norm and Lemma \ref{harmonicvariancespherelemma} yield that for any real $t>0$,
\begin{equation*}
 \begin{split} \sqrt{\mathbb{E} W_2^2 (X, \mathrm{Vol}/(4 \pi))} &\le 2^{1/2} t^{1/2} + 2 \left( \sum_{\ell =1}^{\infty} \frac{e^{-\ell (\ell +1)t}}{\ell (\ell +1)} \sum_{m=-\ell}^{\ell} \mathbb{E} \left| \frac{1}{N} \sum_{n=1}^N Y_{\ell}^m (X_n) \right|^2 \right)^{1/2} \\ &\le 2^{1/2} t^{1/2} + \frac{2^{7/4} 3^{1/2}}{\pi N^{3/4}} \left( \sum_{\ell =1}^{\infty} e^{-\ell (\ell +1)t} \right)^{1/2} . \end{split} 
 \end{equation*}
Clearly, $F(x) := \left| \left\{ \ell \in \mathbb{N} \, : \, \ell (\ell +1) \le x \right\} \right| \le x^{1/2}$, which leads to the estimates
\begin{equation}\label{traceestimate}
\sum_{\ell =1}^{\infty} e^{-\ell (\ell +1)t} = \int_{0}^{\infty} e^{-tx} \, \mathrm{d} F(x) = \int_{0}^{\infty} F(x) t e^{-tx} \, \mathrm{d} x \le t \int_{0}^{\infty} x^{1/2} e^{-tx} \, \mathrm{d} x = \frac{\pi^{1/2}}{2 t^{1/2}} ,
\end{equation}
and
\begin{equation*}
 \sqrt{\mathbb{E} W_2^2 (X, \mathrm{Vol}/(4 \pi))} \le 2^{1/2} t^{1/2} + \frac{2^{5/4} 3^{1/2}}{\pi^{3/4} N^{3/4} t^{1/4}} . 
 \end{equation*}
The optimal choice is $t=\frac{3^{2/3}}{2^{1/3} \pi N}$, in which case we finally obtain
\begin{equation*} 
\sqrt{\mathbb{E} W_2^2 (X, \mathrm{Vol}/(4 \pi))} \le \frac{2^{1/3} 3^{4/3}}{\pi^{1/2} N^{1/2}} < \frac{3.08}{N^{1/2}} . 
\end{equation*}
\end{proof}

\subsection{Harmonic ensemble on $\mathbb{T}^2$}

\begin{proof}[Proof of Theorem \ref{thm: 2dimWassersteintheorem} (ii)] Consider the harmonic ensemble $X=X(T,\mathbb{T}^2)=\{ X_1, X_2, \ldots, X_N \}$. The number of points is now $N=(2T+1)^2$. Lemma \ref{harmonicvariancetoruslemma} shows that for any $k \in \mathbb{Z}^2$, $k \neq 0$,
\begin{equation}\label{varianceestimatetorus}
\mathbb{E} \left| \sum_{n=1}^N e^{2 \pi i \langle k,X_n \rangle} \right|^2 \le 2^{1/2}|k| N^{1/2} .
\end{equation}
Indeed, if $|k_1|, |k_2| \le \sqrt{N}$, this follows from
\begin{equation*}
 N-(2T+1-|k_1|)(2T+1-|k_2|) = (2T+1) (|k_1| + |k_2|) - |k_1| \cdot |k_2| \le (2T+1) 2^{1/2}|k|.  
 \end{equation*}
If either $|k_1|>\sqrt{N}$ or $|k_2|>\sqrt{N}$, then $N \le 2^{1/2} |k| N^{1/2}$, and \eqref{varianceestimatetorus} follows once again.

An application of the smoothing inequality \eqref{torussmoothing}, the triangle inequality for the $L^2$ norm and \eqref{varianceestimatetorus} yield that for any real $t>0$,
\begin{equation}\label{smoothingharmonic2dimtorus}
\begin{split} \sqrt{\mathbb{E} W_2^2 (X,\mathrm{Vol})} &\le 2^{1/2} t^{1/2} + 2 \left( \sum_{\substack{k \in \mathbb{Z}^2 \\ k \neq 0}} \frac{e^{-4 \pi^2 |k|^2 t}}{4 \pi^2 |k|^2} \mathbb{E} \left| \frac{1}{N} \sum_{n=1}^N e^{2 \pi i \langle k,X_n \rangle} \right|^2 \right)^{1/2} \\ &\le 2^{1/2} t^{1/2} + \frac{2^{1/4}}{\pi N^{3/4}} \left( \sum_{\substack{k \in \mathbb{Z}^2 \\ k \neq 0}} \frac{e^{-4 \pi^2 |k|^2 t}}{|k|} \right)^{1/2} . \end{split}
\end{equation}
Let $F(x)=| \{ k \in \mathbb{Z}^2 \, : \, 0<|k|^2 \le x \} |$ denote the number of nonzero lattice points in the closed disk centered at the origin of radius $x^{1/2}$.
Let us draw a unit square centered at each of these lattice points. Since these pairwise disjoint squares are all a subset of the disk of radius $x^{1/2} + 2^{-1/2}$, comparing areas shows that
\begin{equation*}
 F(x) \le \pi \left( x^{1/2} + 2^{-1/2} \right)^2 -1 \le \pi x + \left( 2^{1/2} \pi + \frac{\pi}{2} -1 \right) x^{1/2}, \qquad x \in [1,\infty) . 
 \end{equation*}
Clearly, $F(x)=0$ on $[0,1)$. The infinite series in \eqref{smoothingharmonic2dimtorus} can be expressed in terms of $F(x)$ as
\begin{equation*}
 \begin{split} \sum_{\substack{k \in \mathbb{Z}^2 \\ k \neq 0}} \frac{e^{-4 \pi^2 |k|^2 t}}{|k|} = \int_{0}^{\infty} \frac{e^{-4 \pi^2 t x}}{x^{1/2}} \, \mathrm{d} F(x) &= - \int_{0}^{\infty} F(x) \, \mathrm{d} \frac{e^{-4 \pi^2 t x}}{x^{1/2}} \\ &= \int_{1}^{\infty} F(x) \left( \frac{4 \pi^2 t}{x^{1/2}} + \frac{1}{2 x^{3/2}} \right) e^{-4 \pi^2 t x} \, \mathrm{d} x \\ &= \pi t^{1/2} \int_{4 \pi^2 t}^{\infty} F \left( \frac{x}{4 \pi^2 t} \right) \left( \frac{2}{x^{1/2}} + \frac{1}{x^{3/2}} \right) e^{-x} \, \mathrm{d} x. \end{split} 
 \end{equation*}
One readily checks that
\begin{equation*}
 \begin{split} \int_{0}^{\infty} \frac{x}{4 \pi^2 t} \left( \frac{2}{x^{1/2}} + \frac{1}{x^{3/2}} \right) e^{-x} \, \mathrm{d} x &= \frac{1}{2 \pi^{3/2} t}, \\ \int_{4 \pi^2 t}^{\infty} \left( \frac{x}{4 \pi^2 t} \right)^{1/2} \left( \frac{2}{x^{1/2}} + \frac{1}{x^{3/2}} \right) e^{-x} \, \mathrm{d} x &\le \frac{1}{2 \pi t^{1/2}} \left( 2 + \log \frac{1}{4 \pi^2 t} + \frac{1}{e} \right) , \end{split} 
 \end{equation*}
hence
\begin{equation*}
 \sum_{\substack{k \in \mathbb{Z}^2 \\ k \neq 0}} \frac{e^{-4 \pi^2 |k|^2 t}}{|k|} \le \frac{\pi^{1/2}}{2 t^{1/2}} + \left( 2^{1/2} \pi + \frac{\pi}{2} -1 \right) \left( 1+ \frac{1}{2} \log \frac{1}{4 \pi^2 t} +\frac{1}{2e} \right) .
  \end{equation*}
The estimate \eqref{smoothingharmonic2dimtorus} thus simplifies to
\begin{equation*}
  \sqrt{\mathbb{E} W_2^2 (X,\mathrm{Vol})} \le 2^{1/2} t^{1/2} + \frac{2^{1/4}}{\pi N^{3/4}} \left( \frac{\pi^{1/4}}{2^{1/2}t^{1/4}} + \left( 2^{1/2} \pi + \frac{\pi}{2} -1 \right)^{1/2} \left( 1+ \frac{1}{2} \log \frac{1}{4 \pi^2 t} +\frac{1}{2e} \right)^{1/2} \right) . 
  \end{equation*}
The optimal choice is $t=1/(2^{7/3} \pi N)$, which, after some numerical simplification (note that we may assume that $T \ge 1$, i.e.\ $N \ge 9$), yields
\begin{equation*}
 \sqrt{\mathbb{E} W_2^2 (X,\mathrm{Vol})} \le \frac{3}{2^{2/3} \pi^{1/2} N^{1/2}} + \frac{0.7945 (\log N)^{1/2}}{N^{3/4}} \le \frac{1.7462}{N^{1/2}} . 
 \end{equation*}
\end{proof}

\subsection{Spherical ensemble on $\mathbb{S}^2$}

Consider the spherical ensemble $Z= \{ Z_1, Z_2, \ldots, Z_N \}$ on $\mathbb{S}^2$. We start with a variance estimate, and then give the proof of Theorem \ref{thm: 2dimWassersteintheorem} (iii).
\begin{lem}\label{sphericalvariancelemma} For any $\ell >0$,
\begin{equation*}
 \sum_{m=-\ell}^{\ell} \mathbb{E} \left| \sum_{n=1}^N Y_{\ell}^m (Z_n) \right|^2 \le  \frac{\ell (\ell +1) N^{1/2}}{2 \pi} . 
 \end{equation*}
\end{lem}

\begin{proof} Recall that the kernel defining the spherical ensemble satisfies \eqref{sphericalkernel}. An application of \eqref{varianceformula} with $f(x)=Y_{\ell}^m (x)$ thus gives
\begin{equation*}
 \mathbb{E} \left| \sum_{n=1}^N Y_{\ell}^m (Z_n) \right|^2 = \frac{N}{4 \pi} - \int_{\mathbb{S}^2} \int_{\mathbb{S}^2} \frac{N^2}{16 \pi^2} \left( \frac{1+\langle x,y \rangle}{2} \right)^{N-1} Y_{\ell}^m (x) \overline{Y_{\ell}^m (y)} \, \mathrm{d} \mathrm{Vol} (x) \, \mathrm{d} \mathrm{Vol} (y),   
 \end{equation*}
and the addition formula \eqref{addition} yields the explicit formula
\begin{equation}\label{sphericalvariance}
\begin{split} \sum_{m=-\ell}^{\ell} \mathbb{E} \left| \sum_{n=1}^N Y_{\ell}^m (Z_n) \right|^2 &= \frac{(2 \ell +1)N}{4 \pi} - \int_{\mathbb{S}^2} \int_{\mathbb{S}^2} \frac{N^2}{16 \pi^2} \left( \frac{1+\langle x,y \rangle}{2} \right)^{N-1} \frac{2 \ell +1}{4 \pi} P_{\ell} (\langle x,y \rangle) \, \mathrm{d} \mathrm{Vol} (x) \,\mathrm{d} \mathrm{Vol} (y) \\ &= \frac{(2\ell +1)N}{4 \pi} \left( 1- \frac{1}{2} \int_{-1}^1 N \left( \frac{1+t}{2} \right)^{N-1} P_{\ell} (t) \, \mathrm{d} t \right) . \end{split}
\end{equation}

Assume first that $\ell \le N^{1/2} + N^{-1/2} - 1/2$. The explicit formula for the Legendre polynomials \cite[p.\ 297]{AAR}
\begin{equation*}
 P_{\ell} (t) = \sum_{k=0}^{\ell} (-1)^k \binom{\ell}{k} \binom{\ell+k}{k} \left( \frac{1-t}{2} \right)^k 
 \end{equation*}
immediately gives
\begin{equation*}
 \frac{1}{2} \int_{-1}^1 N \left( \frac{1+t}{2} \right)^{N-1} P_{\ell} (t) \, \mathrm{d} t = \sum_{k=0}^{\ell} (-1)^k \frac{\binom{\ell}{k} \binom{\ell +k}{k}}{\binom{N+k}{k}} .  
 \end{equation*}
As a manifestation of the repulsive nature of the system in Fourier space, the $k=0$ term cancels the main term in the variance, and we obtain
\begin{equation*}
 \sum_{m=-\ell}^{\ell} \mathbb{E} \left| \sum_{n=1}^N Y_{\ell}^m (Z_n) \right|^2 = \frac{(2 \ell +1)N}{4 \pi} \sum_{k=1}^{\ell} (-1)^{k+1} \frac{\binom{\ell}{k} \binom{\ell +k}{k}}{\binom{N+k}{k}} . 
 \end{equation*}
Observe that the terms in this sum are nonincreasing in absolute value. Indeed,
\begin{equation*}
 \frac{\binom{\ell}{k+1} \binom{\ell +k+1}{k+1}}{\binom{N+k+1}{k+1}} \le \frac{\binom{\ell}{k} \binom{\ell +k}{k}}{\binom{N+k}{k}} \, \Longleftrightarrow \, (\ell -k) (\ell +k+1) \le (k+1) (N+k+1) , 
 \end{equation*}
and the latter inequality holds for all $1 \le k \le \ell -1$, since the assumption $\ell \le N^{1/2}+N^{-1/2}-1/2$ ensures, in particular, that $\ell^2+\ell \le 2N+6$. Keeping only the $k=1$ term thus leads to the upper bound
\begin{equation*}
 \sum_{m=-\ell}^{\ell} \mathbb{E} \left| \sum_{n=1}^N Y_{\ell}^m (Z_n) \right|^2 \le \frac{(2\ell +1)N}{4 \pi} \cdot \frac{\ell (\ell +1)}{N+1} \le \frac{\ell (\ell +1) N^{1/2}}{2 \pi} . 
 \end{equation*}
This finishes the proof in the case $\ell \le N^{1/2} + N^{-1/2} -1/2$.

Assume next that $\ell > N^{1/2} + N^{-1/2} -1/2$. Bonnet's recursion formula for the Legendre polynomials \cite[p.\ 303]{AAR}
\begin{equation*}
 (\ell +1) P_{\ell +1} (t) = (2 \ell +1) t P_{\ell} (t) - \ell P_{\ell -1} (t) 
 \end{equation*}
shows that
\begin{equation*}
 (2 \ell +1) \int_{-1}^1 t^{n+1} P_{\ell} (t) \, \mathrm{d} t = (\ell +1) \int_{-1}^1 t^n P_{\ell +1} (t) \, \mathrm{d}t + \ell \int_{-1}^1 t^n P_{\ell -1} (t) \, \mathrm{d} t .  
 \end{equation*}
One readily checks that $\int_{-1}^1 t^n P_{\ell} (t) \, \mathrm{d} t \ge 0$ for all integers $n,\ell \ge 0$ by induction on $n$. Hence
\begin{equation*}
 \int_{-1}^1 \left( \frac{1+t}{2} \right)^{N-1} P_{\ell} (t) \, \mathrm{d} t \ge 0, 
 \end{equation*}
and \eqref{sphericalvariance} leads to the trivial estimate (cf.\ uniformly distributed i.i.d.\ points)
\begin{equation*}
  \sum_{m=-\ell}^{\ell} \mathbb{E} \left| \sum_{n=1}^N Y_{\ell}^m (Z_n) \right|^2 \le \frac{(2\ell +1)N}{4 \pi} . 
  \end{equation*}
This trivial estimate together with the assumption $\ell > N^{1/2} + N^{-1/2}-1/2$ suffices for the claim of the lemma.
\end{proof}

\begin{proof}[Proof of Theorem \ref{thm: 2dimWassersteintheorem} (iii)] An application of the smoothing inequality \eqref{spheresmoothing}, the triangle inequality for the $L^2$ norm, Lemma \ref{sphericalvariancelemma}, and estimate \eqref{traceestimate} yield that for any real $t>0$,
\begin{equation*}
 \begin{split} \sqrt{\mathbb{E} W_2^2 (Z,\mathrm{Vol}/(4 \pi))} &\le 2^{1/2} t^{1/2} + 2 \left( \sum_{\ell =1}^{\infty} \frac{e^{-\ell (\ell +1)t}}{\ell (\ell +1)} \sum_{m=-\ell}^{\ell} \mathbb{E} \left| \frac{1}{N} \sum_{n=1}^N Y_{\ell}^m (Z_n) \right|^2 \right)^{1/2} \\ &\le 2^{1/2} t^{1/2} + \frac{2^{1/2}}{\pi^{1/2}N^{3/4}} \left( \sum_{\ell =1}^{\infty} e^{-\ell (\ell +1)t} \right)^{1/2} \\ &\le 2^{1/2} t^{1/2} + \frac{1}{\pi^{1/4} N^{3/4}t^{1/4}} . \end{split} 
 \end{equation*}
The optimal choice is $t=1/(4 \pi^{1/3} N)$, and we finally obtain
\begin{equation*}
 \sqrt{\mathbb{E} W_2^2 (Z,\mathrm{Vol}/(4 \pi))} \le \frac{3}{2^{1/2} \pi^{1/6} N^{1/2}} < \frac{1.7529}{N^{1/2}} . 
 \end{equation*}
\end{proof}

\subsection{One-dimensional torus}\label{sec:onedimtorus}

Using a connection between the Wasserstein metric and negative Sobolev norms, Graham \cite[Remark 7]{GR} established an explicit formula for $W_2$ on the $1$-dimensional torus: for any finite set $A_N = \{ a_1, a_2, \ldots, a_N \} \subset \mathbb{T}$,
\begin{equation*}
 W_2^2 (A_N,\mathrm{Vol}) = \sum_{\substack{k \in \mathbb{Z}\\ k \neq 0}} \frac{1}{4 \pi^2 k^2} \left| \frac{1}{N} \sum_{n=1}^N e^{2 \pi i k a_n} \right|^2 . 
 \end{equation*}
We emphasize that this formula holds only in dimension $d=1$, and only for the distance from the Riemannian volume. In light of formula \eqref{L2discrepformula}, the distance from uniformity in $W_2$, the periodic $L^2$-discrepancy and the diaphony are thus identical up to normalizing constants. In particular, $W_2(A_N,\mathrm{Vol}) = 2^{-1/2} D_{\mathrm{per},2}(A_N)$.
\begin{proof}[Proof of Theorem \ref{thm:1dimWasserstein}] The claim now immediately follows from formula \eqref{onedimL2discrep}.
\end{proof}

\section*{Acknowledgments}

Bence Borda was supported by the Austrian Science Fund (FWF) project M 3260-N. Peter Grabner and Ryan Matzke were supported by the Austrian Science Fund FWF project F5503 part of the Special Research Program (SFB) ``Quasi-Monte Carlo Methods: Theory and Applications.'' Ryan Matzke was also supported by the NSF Postdoctoral Fellowship Grant 2202877. This work was finished during a visit of Peter Grabner to Vanderbilt University, which was funded by Doug Hardin. He would like to extend his gratitude for the hospitality. The authors would also like to thank Carlos Beltr\'{a}n for his helpful communications, as well as the two anonymous referees for their helpful comments and suggestions.

\end{document}